\documentclass[12pt,a4paper]{amsart}
\usepackage{amsmath,amssymb, amsbsy}
\usepackage{color,psfrag}
\usepackage[dvips]{graphicx}
\usepackage[dvipsnames]{xcolor}
\usepackage{enumerate}
\usepackage{stackrel}
\usepackage[bookmarks=false]{hyperref}

\usepackage{mathtools}
\mathtoolsset{showonlyrefs}

\renewcommand{\P}{\mathbb{P}}

\newcommand{\R}{{\mathbb{R}}}

\newcommand{\N}{{\mathbb N}}

\newcommand{\ieq}{\begin{equation}}
\newcommand{\eeq}{\end{equation}}
\newcommand{\ieqa}{\begin{eqnarray}}
\newcommand{\eeqa}{\end{eqnarray}}
\newcommand{\ieqas}{\begin{eqnarray*}}
\newcommand{\eeqas}{\end{eqnarray*}}

\theoremstyle{plain}
\newtheorem{theorem}{Theorem} [section]
\newtheorem{corollary}[theorem]{Corollary}
\newtheorem{lemma}[theorem]{Lemma}

\def\neweq#1{\begin{equation}\label{#1}}
\def\endeq{\end{equation}}

\theoremstyle{definition}
\newtheorem{definition}[theorem]{Definition}
\newtheorem{remark}[theorem]{Remark}

\numberwithin{figure}{section}

\def\F{{\mathcal F}}

\def\F{{\mathcal F}}
\def\G{{\mathcal G}}

\newcommand{\average}{{\mathchoice {\kern1ex\vcenter{\hrule
height.4pt width 6pt depth0pt} \kern-11pt} {\kern1ex\vcenter{\hrule height.4pt width 4.3pt depth0pt} \kern-7pt} {} {} }}

\usepackage{graphicx}
\usepackage{subfigure}
\graphicspath{
{Images/}
{Images/ImagesInsiderRan/}
{Images/ImagesCEC/}
{ImagesAPO/}
}

\begin{document}

\title[]{Optimal investment with insider information using Skorokhod \& Russo-Vallois integration}

\author[M. Elizalde, C. Escudero, T. Ichiba]{Mauricio Elizalde, Carlos Escudero, Tomoyuki Ichiba}
\address{}
\email{}


\maketitle


\vskip2mm
\noindent
{\footnotesize

Mauricio Elizalde:
Departamento de Matem\'aticas Fundamentales.
Universidad Nacional de Educaci\'on a Distancia.
ORCID 0000-0002-3275-5884,
{\tt melizalde@mat.uned.es}

Carlos Escudero:
Departamento de Matem\'aticas Fundamentales.
Universidad Nacional de Educaci\'on a Distancia.
ORCID 0000-0002-9597-0506.
{\tt cescudero@mat.uned.es}

Tomoyuki Ichiba:
Department of Statistics \& Applied Probability.
University of California Santa Barbara.
ORCID 0000-0003-2787-7569.
{\tt ichiba@pstat.ucsb.edu}
\vskip3mm\noindent
}






{\bf Abstract.} We study the maximization of the logarithmic utility for an insider with different anticipating techniques. Our aim is to compare the utilization of Russo-Vallois forward integral and Skorokhod integral in this context.
Theoretical analysis and illustrative numerical examples showcase that the Skorokhod insider outperforms the forward insider. This remarkable observation stands in contrast to the scenario involving risk-neutral traders. Furthermore, an ordinary trader could surpass both insiders if a significant negative fluctuation in the driving stochastic process leads to a sufficiently negative final value. These findings underline the intricate interplay between anticipating stochastic calculus and nonlinear utilities, which may yield non-intuitive results from the financial viewpoint.


\section{Introduction}
\label{Introduction}

The development of investment strategies with insider information is an ongoing topic of financial mathematics (\cite{pikovsky1996anticipative}, \cite{imkeller2001free}, \cite{leon2003anticipating}, \cite{corcuera2004}, \cite{biagini2005}, \cite{kohatsu2007}, \cite{di2009malliavin},  \cite{draouil2015donsker}), and it is strongly connected to advancements in the stochastic analysis theory, like Malliavin calculus (\cite{nualart2006malliavin}, \cite{nualart2018introduction}), or anticipative integration and anticipative transformations (\cite{skorokhod1976}, \cite{buckdahn1989}, \cite{russo1993forward}, \cite{buckdahn1994}).
We aim to explore various interpretations of noise within the insider wealth dynamics and subsequently compare the outcomes to determine which interpretation aligns more closely with economic viability, and, in more generality, to derive consequences for financial modeling.

In this work, we compare the usage of Skorokhod \cite{skorokhod1976} and Russo-Vallois forward integration \cite{russo1993forward} in the situation that a trader has insider information about the future price of a given stock and desires to maximize her expected utility under a logarithmic risk aversion. She invests in a portfolio consisting of the stock and a risk-free asset until the time horizon $T$, which corresponds to the time of the privileged information. We assume the trader cannot influence the market prices.

We are inspired by the work of Escudero \cite{escudero2018}, who addresses the problem of insider trading with one-period investment and without considering any utility function to model risk aversion (i.e., the traders are assumed to be risk-neutral). The author in \cite{escudero2018} compares the usage of Skorokhod and Russo-Vallois forward integration and concludes that the usage of the Russo-Vallois forward integral is more meaningful from the financial point of view because the expected utility of the insider trader under Skorokhod integration is less than that of an ordinary trader. On the contrary, the expected wealth under forward integration is bigger than that of the ordinary trader. By ordinary trader, we mean a trader who has no more information than the present and historical prices of the stock. Obviously, this work leaves open the case of a nonlinear utility; and this is precisely the question we will address herein:
how the logarithmic utility interacts with both types of stochastic calculi. Surprisingly, the present results differ from those in \cite{escudero2018}.

The outline of the paper is as follows. First, we develop the case in which the trader knows the exact value of the driving process of the stock price at the horizon time.  Karatzas and Pikovsky \cite{pikovsky1996anticipative} face this problem using a Brownian bridge with It\^{o} integration and enlargement of filtration. In this work, we also start with an example of how to handle the problem using a Brownian bridge with It\^{o} integration in Section \ref{APO_BB}. This case is devoted to show the consistency of our approach. Then, we continue with the usage of Russo-Vallois forward integration, based in {\O}ksendal and R{\o}se
work \cite{oksendal2017}, in Section \ref{APO_fw}, and Skorokhod integration in Section \ref{APO_sk}. To handle the solution of the related stochastic differential equation in the Skorokhod scheme, we consider the anticipative Girsanov transformations studied by Buckdahn in \cite{buckdahn1989} and \cite{buckdahn1994}. We exemplify an investment assuming that a trader has insider information about the prices of the 2-Year U.S. Treasury Note Future, and conclude that with Skorokhod integration, the trader has more expected wealth. We also illustrate this fact with simulations of the Brownian paths that drive the risky asset and comparing the wealth of investments under these two types of anticipative integration.

Later on, in Section \ref{Distribution} we compare the expected differences in wealth of the investor when using forward and Skorokhod schemes in the presence of some uncertainty. We find that the expected value of the wealth under Skorokhod integration is bigger than its expected value under Russo-Vallois forward integration. Therefore, all of our current results point to the opposite direction than those in \cite{escudero2018} (with, of course, different
hypotheses on risk aversion).

For the setting, we work on a probability space, $(\Omega,\mathcal{F},\P)$, equipped with $\F=\{\mathcal{F}_t\}_{0\leq t\leq T}$, the natural filtration of the Brownian motion $W_t$, $0\leq t\leq T$ for some $T>0$. The investment consists of a portfolio with a risk-free asset $R_t$ modeled by
\begin{equation}
dR_t=R_t  r_t  dt, \ \ R_0>0, \ \ t\in[0,T],
\label{risk-free_model}
\end{equation}
where $r_t$ is the risk-free instantaneous rate; and the risky asset $S_t$, for which the trader has the mentioned information, modeled by a geometric Brownian motion
\begin{equation}
dS_t=S_t \mu_t dt + S_t \sigma_t dB_t,\ \ S_0>0,\ \ t\in[0,T],
\label{risky_model}
\end{equation}
where $\mu_t$ is the appreciation rate and $\sigma_t$ is the volatility of this risky asset, and $B_t$ is a suitable driving process of $S_t$, that depends on the standard Brownian motion $W_t$. For each method we consider, we make a suitable choice for $B_t$. In the case of the ordinary trader, we simply use $B_t = W_t$. The alternative choice for the insider will be justified in section \ref{APO_BB}. We denote the proportion of the total wealth of the trader $X_t$ invested in the stock at time $t$ by $\pi_t$; hereafter we will refer to this proportion as the ``portfolio''. Therefore, as a consequence of the self-financing condition,
the stochastic differential equation (SDE) for the wealth process of the trader is
\begin{equation}
\begin{array}{ll}
	dX_t &= (1-\pi_t) X_t \ r_t \ dt + \pi_t ( X_t \ \mu_t \ dt + X_t \ \sigma_t \ dB_t ) \\
	&= X_t(\mu_t \ \pi_t + r_t (1-\pi_t)) dt + \sigma_t \ \pi_t X_t \ dB_t,
\end{array}
\label{SDE}
\end{equation}

\noindent
with initial condition $X_0 \in \mathbb{R}^+$ ($\mathbb{R}^+:=(0, \infty)$). We will eventually assume that no short selling is allowed, i.e. the value of $\pi_t$ is between 0 and 1, since
this condition will be necessary in order to find an interpretation of the optimal portfolio for the Skorokhod integral, although this is not needed for the Russo-Vallois forward integral.

We denote by $X_t^{\pi}$ the wealth of the trader under the portfolio $\pi$. Our goal is to find the optimal portfolio $\pi_t^*$ that maximizes the expected logarithmic terminal wealth at time $T$,
\begin{equation}
\pi_t^* := \arg\max \mathbb{E} \left[ \log X^{\pi}_{T}\right].
\label{pi}
\end{equation}

In his celebrated work \cite{merton1969}, Merton shows, with It\^o integration, that if the driving process is the standard Brownian motion $W_t$ and without more information than the historical prices, i.e. under the filtration $\mathcal{F}_t, \ 0\leq t  \leq T$, the optimal value $\pi_t^*$ that maximizes the expected logarithmic terminal wealth, $\log X_{T}^{\pi}$, is
\begin{equation}
\pi_t^*= \dfrac{\mu_t - r_t}{\sigma_t^2},
\label{pi_Merton}
\end{equation}
making the value of the optimization problem to be
$$
\begin{array}{ll}
V_{T}^{\pi^*}
&:= \mathbb{E} \left[ \log (X^{\pi^*}_{T} / X_0) \right] \\\\
&= \mathbb{E}\displaystyle\int_0^{T} \left[\mu_t\pi_t^* + r_t (1-\pi_t^*)-\dfrac{1}{2}\sigma_t^2\pi_t^{*^2} \right] dt \\\\
&= \mathbb{E}\displaystyle\int_0^{T} \left[ r_t + \dfrac{1}{2}  \left(  \dfrac{\mu_t-r_t}{\sigma_t} \right)^2 \  \right] dt,
\end{array}
$$
under appropriate conditions on $\pi_t$, $\mu_t$, $r_t$, and $\sigma_t$.
This result gives us a benchmark to compare the value of the problem for an ordinary trader, who does not have privileged information, with the value of an insider trader. For the sake of simplicity, from now on we will always assume $X_0 \in \mathbb{R}^+$ (i.e. it is a positive constant).

\section{Stochastic Analysis for Anticipative Processes}
\label{Stochastic}

In this section, we provide the definitions and results within the field of stochastic analysis we need to apply for the portfolio optimization with anticipating information.

For a given time horizon $T>0$, we work on a Wiener space $(\Omega, \mathcal{F},\mathbb{P})$ over the sample space of continuous functions over $[0,T]$, where $\mathcal{F}\ $ is the smallest Borel sigma-algebra that contains $\Omega$, and $\mathbb{P}$ is a Wiener measure under which the canonical process $W_t(\omega)=\omega(t)=\omega_t$, $0\leq t \leq T$, is a standard Brownian motion. We let $L_2(\Omega)$ denote the space of the square-integrable random variables on $\Omega$.

\subsection{The Malliavin Derivative}
\label{Malliavin}
$ $

Let $\mathcal{S}$ be the space of smooth Wiener functionals in the sense that if a random variable $F$ belongs to $\mathcal{S}$, there exists $n\in\N$ and $n$ time points $t_1,...,t_n$ with   $0\leqslant t_1,...,t_n\leqslant T$ and a smooth bounded function $ f\in C^{\infty}(\R^n)$  such that $F$ is represented as  $F=f(W_{t_1},...,W_{t_n})=f(\omega_{t_1},...,\omega{t_n})$.

For every smooth Wiener functional $F$ in $\mathcal{S}$, we define the unbounded linear operator $D: L_2(\Omega)\rightarrow L_2([0,T]\times\Omega)$, as in \cite{buckdahn1994}, given by
\begin{equation}
D_t F = \sum^n_{i=1} \dfrac{\partial f}{\partial x_i}(\omega_{t_1},...,\omega_{t_n})\cdot \mathbb{I}_{[0,t_i)}(t),\  0\leqslant t_k \leqslant T,
\label{Malliavin_derivative}
\end{equation}
where $\mathbb{I}_A(\cdot)$ is the characteristic function of the set $A$ such that $\mathbb{I}_A (t) = 1$ if $t \in A$ and $\mathbb{I}_A (t) = 0$ otherwise. $D_t F$ is called the Malliavin derivative of $F$ at $(t,\omega) \in [0,T]\times \Omega$. In general, we define the $k$-th derivative of $F$,
for $k \geq 1$, $0 \leq s_1,...,s_k \leq T$, as:

$$D^k_{s_1,...,s_k} F = D_{s_1},...,D_{s_k} F.$$

The mapping $D$ is a closable unbounded linear operator from $L^2(\Omega)$ into $L^2([0,T]\times \Omega)$ (see \cite{buckdahn1994}). We identify $D$ with its closed extension and denote its domain by $\mathbb{D}_{1,2}$. For any $k\geqslant 1$, $2 \leqslant p < \infty$, we introduce the spaces $\mathbb{D}_{k,p}$ as the closure of $\mathcal{S}$ with respect to the Sobolev norm (\cite{nualart2006malliavin} chapter 1.5)
$$
\parallel F \parallel_{k,p} \,\,\, := \,\,\, \parallel F \parallel_p + \Big \lVert \left( \int_{[0,T]^k} \mid D^k_z F \mid^2 dz \right)^{1/2} \Big \rVert_p , \ F \in \mathcal{S},
$$
where $\parallel \cdot \parallel_p$ is the $L^p(\Omega)$ norm.
The concept of the Malliavin derivative leads us to define the Skorokhod integral in the following section.

\subsection{The Skorokhod Integral}
\label{Skorokhod}
$ $

Skorokhod (\cite{skorokhod1976}) introduced in 1976 a generalization of the It\^o integral that coincides with the adjoint operator $\delta: L_2([0,T]\times\Omega)\rightarrow L_2(\Omega)$ of the derivative operator $D$ in the following sense: the domain Dom($\delta$) of the operator $\delta$ is the set of processes $u\in L_2([0,T]\times\Omega)$ for which there exists a random variable $G^u \in L_2(\Omega)$ satisfying the adjoint relationship
\begin{equation}
\mathbb{E}[G^uF] = \mathbb{E} \Big[\int_0^{T} u_sD_sF ds \Big],
\label{Skorokhod_relationship}
\end{equation}
for all $F\in \mathcal{S}$. The random variable $G^u$ is uniquely determined in $L_2(\Omega)$ for every $u$, and is called the Skorokhod integral of $u_. \in \textup{Dom }(\delta)$, and denoted by
$\delta(u):=G^u$. We also use the notation:
\begin{equation}
\delta (u)
:= \int_0^T u_s \delta W_s.
\end{equation}
Note that the left-hand side of the equality in the adjoint relationship is the inner product of $L^2(\Omega)$, so the Skorokhod integral is non-local in $\Omega$; nevertheless $\delta(u)$ is a well-defined squared integrable random variable, and hence it is well-defined almost surely. It is worth to mention that this integral admits a characterization in terms of a Riemann sum approximation; see section 3.1.1 of \cite{nualart2006malliavin} for the details.

\subsection{Anticipative Girsanov Transformations}
\label{Anticipative}
$ $

The anticipative Girsanov transformations allow us to solve stochastic differential equations of the form
\begin{equation}
X_t = X_0 + \int_0^t \hat{\mu}_s X_s ds + \int_0^t \hat{\sigma}_s X_s \delta W_s,\ \ 0\leqslant t\leqslant T,
\label{integral_SDE_sk}
\end{equation}
where $\delta$ denotes Skorokhod integration. As always, $X_0 \in \mathbb{R}^+$.

In \cite{buckdahn1989}, Buckdahn proves that, under the assumptions $\hat{\sigma_t}\in L_\infty ([0,T])$ and $\hat{\mu_t}\in L_\infty ([0,T]\times \Omega)$, $ 0\leqslant t\leqslant T$, equation \eqref{integral_SDE_sk} admits a unique solution in the following sense:
%
\begin{itemize}
    \item [(i)]$\mathbb{I}_{[0,t]}\hat{\sigma}X \in \textup{Dom}(\delta),\ $ and
    \item [(ii)] the Skorokhod integral satisfies $\mathbb{P}$-a.s.
    $$\delta(\mathbb{I}_{[0,t]}\hat{\sigma}X) =
    \displaystyle\int_0^t\hat{\sigma}_s X_s \delta W_s = X_t - X_0-\int_0^t\hat{\mu}_s X_sds;\ \ 0\leqslant t\leqslant T.$$
\end{itemize}

The derivation of this solution is based on a change of measure technique adapted to anticipative settings. Below, we outline the transformation used in this approach.

For a deterministic process $\hat{\sigma_t}\in L_\infty ([0,T])$, we define the family of transformations $U_{s,t} : \Omega \rightarrow \Omega, \ 0\leqslant s \leqslant t \leqslant T$, of $\omega\in\Omega$, shifted with respect to $\mathbb{I}_{[s,t]}(r)\cdot \hat{\sigma}_r$, given by
\begin{equation}
U_{s,t} : \{\omega_v, 0\leqslant v\leqslant T\} \longmapsto \Bigg\{ U_{s,t} \ \omega_v:= \omega_v - \displaystyle\int_0^v \mathbb{I}_{[s,t]}(r) \hat{\sigma}_r dr,\ 0\leqslant v\leqslant T\Bigg\}.
\label{U_st}
\end{equation}
This transformation corresponds to a pathwise shift of the Brownian trajectories, which allows the stochastic integrals to admit a well-defined solution in the Skorokhod sense, as established in~\cite{buckdahn1989}.

We let $T_{s,t}$ denote the inverse transformation of $U_{s,t}$, given by
\begin{equation}
\omega. = T_{s,t} \circ U_{s,t} \ \omega. = (U_{s,t} \ \omega). + \displaystyle\int_0^\cdot \mathbb{I}_{[s,t]}(r) \cdot \hat{\sigma}_r dr,\ \omega\in\Omega,
\label{U_inverse}
\end{equation}
for every fixed $0\leqslant s\leqslant t\leqslant T.$ For ease of notation, we write $U_t = U_{0,t},\  T_t = T_{0,t}$. And we use $T_s U_t \ \omega= U_{s,t} \ \omega$ for $\hat{\sigma_t} \in L_\infty ([0,T])$, $0\leqslant s\leqslant t\leqslant T$, since
$$
\begin{array}{lll}
T_sU_t \ \omega. &= T_s \Bigg( \omega. - \displaystyle\int_0^\cdot \mathbb{I}_{[0,t]} (r) \hat{\sigma}_r dr \Bigg) = \Bigg[ \omega. - \int_0^\cdot \mathbb{I}_{[0,t]} (r) \hat{\sigma}_r dr \Bigg] + \int_0^\cdot \mathbb{I}_{[0,s]} (r) \hat{\sigma}_r dr
\\\\
&= \omega. - \displaystyle\int_0^\cdot \mathbb{I}_{[s,t]} (r) \hat{\sigma}_r dr - \int_0^\cdot \mathbb{I}_{[0,s]} (r) \hat{\sigma}_r dr + \int_0^\cdot \mathbb{I}_{[0,s]} (r) \hat{\sigma}_r dr
\\\\
&=U_{s,t}  \ \omega.
\end{array}
$$

The solution of  \eqref{integral_SDE_sk}, given by Buckdahn in \cite{buckdahn1989}, is represented by
\begin{equation}
X_t = X_0 \cdot \exp\Bigg\{\int_0^t\hat{\mu}_s(U_{s,t})ds\Bigg\}L_t,\ \mathbb{P}- \textup{ a.s., } 0\leqslant t\leqslant T,\\ \label{sol_sk}
\end{equation}
where $L_t = \exp\Bigg\{ \displaystyle\int_0^t \hat{\sigma}_s \delta W_s - \dfrac{1}{2} \int_0^t\hat{\sigma}_s^2 ds \Bigg\}$.

Later, in \cite{buckdahn1994}, Buckdahn defines, for stochastic $\hat{\sigma}_\cdot$, the transformation
 $U'_{s,t} : \Omega \rightarrow \Omega$,  $0\leqslant s \leqslant t \leqslant T$, of $\omega\in\Omega$, shifted with respect to $\mathbb{I}_{[s,t]}(r)\cdot \hat{\sigma}_r(U'_{r,t} \ \omega)$, given by
\begin{equation}
U'_{s,t} \ \omega. = \omega. - \int_0^{\cdot} \mathbb{I}_{[s,t]}(r)\cdot \hat{\sigma_r} (U'_{r,t} \ \omega.)dr,
\label{U_st_stochastic}
\end{equation}
and shows that the linear SDE (\ref{integral_SDE_sk}) with $\hat{\sigma_t}\in L_2 ([0,T]\times \Omega)$ and $\hat{\mu_t}\in L_\infty ([0,T]\times \Omega)$,
has the unique solution
\begin{equation}
X_t = X_0 \cdot \exp\Bigg\{\int_0^t \hat{\mu}_s (U'_{s,t})ds \Bigg\}L'_t\  , \ \mathbb{P}- \textup{ a.s., } 0\leqslant t\leqslant T,\\
\label{sol_sk_stochastic}
\end{equation}
where
\begin{equation}
\begin{array}{ll}
L'_t &=  \exp\Bigg\{ \displaystyle\int_0^t \hat{\sigma}_s (U'_{s,t}) \delta W_s - \frac{1}{2} \int_0^t \hat{\sigma}_s (U'_{s,t})^2 ds
\\
&- \displaystyle\int_0^t \int_s^t (D_u \hat{\sigma}_s) (U'_{s,t}) D_s [\hat{\sigma}_u (U'_{s,t})] \ du \ ds \Bigg\}.
\end{array}
\label{Lt_stochastic}
\end{equation}

The last expressions are closed for deterministic $\hat{\sigma_t}$, but not for stochastic $\hat{\sigma_t}$.

\subsection{The Russo-Vallois Forward Integral}
\label{Forward}
$ $

Russo and Vallois define in 1993 a forward integral with respect to Brownian motion
by an approximation procedure \cite{russo1993forward}.

\begin{definition}
A stochastic process $\phi_t,\ t\in [0,T] $, is said to be forward integrable in the weak sense with respect to a standard Brownian motion $W_t$,
if there exists another stochastic process $I_t$ such that
\begin{equation}
\stackbin[0\leqslant t \leqslant T]{}{\sup} \ \Bigg\vert \int_0^t \phi_s \, \dfrac{W_{s+\epsilon}-W_s}{\epsilon}\ ds - I_t\Bigg\vert \rightarrow 0\ ,\ \ \epsilon\rightarrow 0^+
\label{Forward_integral}
\end{equation}
in probability. If such a process exists, we denote
$$
I_t:= \int_0^t \phi_s \, d^-W_s,\ t\in[0,T],
$$
the Russo-Vallois forward integral of $\phi_t$ with respect to $W_t$ over $[0,T]$.
\end{definition}

This forward integral is an extension of the It\^o integral. If $\phi$ is adapted to the filtration $\mathcal F_t$ and It\^o integrable, then $\phi$ is forward integrable and its forward integral coincides with its It\^o integral. The proof of this statement is in \cite{russo1993forward}.

A Russo-Vallois forward process (with respect to $W_t$) is a stochastic process of the form
$$X_t = x + \int_0^t u_sds + \int_0^t v_sd^-W_s,\ \ t\in[0,T] ,$$
where $\int_0^T \vert u_t\vert ds<\infty \, a.s.$ and $v$ is a forward integrable stochastic process. A shorthand notation for this is
$$d^-X_t = u_t dt + v_td^-W_t. $$

We present the It\^o formula for Russo-Vallois forward integrals as stated in \cite{di2009malliavin}, page 136.
See also \cite{russo1995generalized} and \cite{flandoli2002generalized}.

\begin{theorem}
\label{ito_lemma_fw}
Let $X_t$, $t\in[0,T]$, be a Russo-Vallois forward process defined as above and  let $\mathit{f}\in C^{1,2} \left( [0,T]\times \mathbb{R} \right)$. Define $Y(t)=\mathit{f}(t,X_t)$. Then, $Y_t$ is a forward process and
$$d^-Y_t = \dfrac{\partial\mathit{f}}{\partial t}(t,X_t)dt\ +\ \dfrac{\partial\mathit{f}}{\partial x}(t,X_t)d^-X_t\ +\ \dfrac{1}{2}\dfrac{\partial^2\mathit{f}}{\partial x^2}(t,X_t)v^2_t dt.$$
\end{theorem}

In order to have a Riemann sum interpretation of the Russo-Vallois forward integral, take a partition $J_n$ of $0 = t_0<t_1 < ... < t_{J_n} = T$ of $[0,T]$. Assume $\varphi$ is \textit{c\`agl\`ad} (i.e. left continuous with right limits) and forward integrable, and moreover a simple stochastic process, meaning that
$$
\varphi(t) = \sum^{J_n}_{j=1} \varphi (t_{j-1})	\chi_{(t_{j-1},t_j]}(t),\ \ t \in [0,T],
$$
then, the following identity
$$
\int^T_0 \varphi (s)d^- W(s) = \stackbin[\Delta t\rightarrow 0]{}{\lim} \sum^{J_n}_{j=1} \varphi (t_{j-1})(W(t_j)-W(t_{j-1})),
$$
holds with convergence in probability, where $\Delta t := \max_{j=1,...,J_n} (t_j - t_{j-1}) \longrightarrow 0$ as $\ n \longrightarrow \infty$.
For details, see, for instance, reference \cite{biagini2005}.

Based on the previous convergence, when the integrand is adapted, the Riemann sums serve as an approximation to the It\^o integral with respect to
Brownian motion. Consequently, the Russo-Vallois forward integral and the It\^o integral coincide. Therefore, we can interpret this integral as an extension of the It\^o integral to the anticipating context. Note that this property is shared with the Skorokhod integral, although it constitutes a different extension.


\noindent\textbf{Relationship between Forward and Skorokhod Integration.}

There is a relation between the Russo-Vallois forward integral and Skorokhod integral that allows to compute forward integrals in terms of the Skorokhod integrals and
Malliavin derivatives. To make it explicit, we follow Chapter 8 in \cite{di2009malliavin} to introduce the definition of the forward integral in the strong sense and the class of stochastic processes $\mathbb{D}_0$. Subsequently, we state this relation along with a consequence of it.

\begin{definition}\label{defdo}
The class $\mathbb{D}_0$ consists of all measurable stochastic processes $\varphi$ such that

\begin{itemize}
\item[1)] the trajectories $\varphi(\cdot , \omega):t\longrightarrow \varphi(t , \omega)$ are \textit{c\` agl\` ad} a.s.

\item[2)] the random variables $\varphi (t) \in \mathbb{D}_{1,2}$ for all $t\in [0,T]$.

\item[3)] the trajectories $t \longrightarrow D_s\varphi (t)(\omega)$ are \textit{c\` agl\` ad} for almost every $s \in [0,T]$ a.s.

\item[4)] the limit $D_{t^+}\varphi(t):=\lim_{s\rightarrow t^+} D_s \varphi(t)$ exists with convergence in $L^2(\mathbb{P})$.

\item[5)] $\varphi$ is Skorokhod integrable.
\end{itemize}
\end{definition}

\begin{definition}
A stochastic process $\varphi_t,\ t\in [0,T] $, is said to be forward integrable in the strong sense with respect to a standard Brownian motion $W_t$ if the limit
$$
\lim_{\varepsilon\rightarrow 0^+} \int_0^T \varphi (t) \dfrac{W(t+\varepsilon) - W(t)}{\varepsilon} dt
$$
exists in $L^2(\mathbb{P})$.
\end{definition}

The above mentioned relation reads as follows:

\begin{theorem} Let $\varphi$ be a process in $\mathbb{D}_0$. Then, $\varphi$ is forward integrable  in the strong sense and moreover
$$
\int_0^T \varphi(t)d^- W(t) = \int_0^T \varphi(t)\delta W(t) + \int_0^T D_{t^+} \varphi(t) dt.
$$
\end{theorem}

By the zero-mean property of the Skorokhod integral we get the immediate consequence:

\begin{corollary} Let $\varphi$ be a process in $\mathbb{D}_0$. Then
$$
\mathbb{E} \left[ \int_0^T \varphi(t)d^-W(t) \right] = \mathbb{E} \left[ \int_0^T D_{t^+} \varphi(t) dt \right].
$$
\label{corollary_forward}
\end{corollary}

\subsection{Donsker Delta Function}
\label{Donsker}
$ $

In this section, we define the space of Hida distributions, which is needed to define the Donsker delta function.
To this end, we follow Chapter 6 of \cite{di2009malliavin}.

We start defining the Hermite polynomials $h_n(x)$ of degree $n$, which read
$$
h_n(x) := (-1)^n e^{\frac{1}{2}x^2} \dfrac{d^n}{dx^n} \left( e^{-\frac{1}{2}x^2} \right),\ n=0,1,2,... .
$$
Let $e_k$ be the $k$-th Hermite function defined by
$$
e_k(x):= \pi^{-\frac{1}{4}}\left(\left( k-1 \right)!\right)^{-\frac{1}{2}}e^{-\frac{1}{2}x^2} h_{k-1}\left( \sqrt{2}x \right),\ \ k=1,2,...,
$$
and additionally define
$$
\theta_k(\omega):=\langle \omega,e_k \rangle = \int_{\mathbb{R}} e_k(x)dW(x,\omega),\ \ \omega\in\Omega.
$$
Let $\mathcal{J}$ denote the set of all finite multi-indices $\alpha=\left( \alpha_1,\alpha_2,...,\alpha_m \right),\ m=1,2,...,$ of non-negative integers $\alpha_i$. We set
$$
H_{\alpha}(\omega):= \prod^m_{j=1} h_{\alpha_{j}}\left( \theta_j(\omega) \right),\ \ \omega\in\Omega,
$$
and $H_0:=1$, for $\alpha=\left( \alpha_1,...,\alpha_m \right)\in\mathcal{J},\ \alpha\neq 0$.


Let $S = S(\mathbb{R}^d)$ be the Schwartz space of rapidly decreasing $C^{\infty}(\mathbb{R}^d)$ real functions on $\mathbb{R}^d$. We define the Hida test function space $(S)$ as the space
$$
(S) = \bigcap_{k \in \mathbb{R}} (S)_k,
$$
where $f = \displaystyle\sum_{\alpha\in\mathcal{J}} a_\alpha H_\alpha \in L^2(P)$, $a_\alpha \in \mathbb{R}$, belongs to the Hida test function Hilbert space $(S)_k$ whenever
$$
\parallel f \parallel^2_k \,\,\, := \sum_{\alpha\in\mathcal{J}} \alpha ! a^2_\alpha (2\mathbb{N})^{\alpha k} < \infty,
$$
where
$$
(2\mathbb{N})^\alpha = \prod^m_{j=1} (2j)^{\alpha_j}, \textup{ for } \alpha = (\alpha_1,...,\alpha_m) \in \mathcal{J}.
$$

We define the Hida distribution space $S^{*}$ as the space:
\[
(S)^{*} = \bigcup_{q\in\mathbb{R}}(S)_{-q},
\]
equipped with the following inductive topology: a sequence $\{F_n\}_{n\in\mathbb{N}}$ converges to $F$ in $(S)^{*}$ if and only if there exists $q\in\mathbb{R}$ such that:
\[
\|F_n - F\|_{-q} \to 0 \quad \text{as } n \to \infty.
\]

For any Hida distribution $F = \sum_{\alpha\in\mathcal{J}} b_\alpha H_\alpha \in (S)^{*}$, we define its generalized expectation as:
\[
E[F] = b_0. 
\]
When $F \in L^{2}(P)$, this generalized expectation coincides with the standard probabilistic expectation, since $E[H_\alpha] = 0$ for all multi-indices $\alpha \neq 0$.

The space $(S)^{*}$ is the topological dual of the Hida test function space $(S)$. The dual pairing between $F = \sum_{\alpha} b_\alpha H_\alpha \in (S)^{*}$ and $f = \sum_{\alpha} a_\alpha H_\alpha \in (S)$ is given by:
\[
\langle F, f \rangle = \sum_{\alpha \in \mathcal{J}} \alpha! \, a_\alpha b_\alpha.
\]

We have the following continuous embeddings:
\[
(S) \subset (S)_k \subset L^{2}(P) \subset (S)_{-q} \subset (S)^{*}, \quad \text{for all } k,q > 0 . 
\]

We have introduced the adequate framework to introduce the Donsker delta function.

\begin{definition}
Let $Y : \Omega\longrightarrow\mathbb{R}$ be a random variable that belongs to the Hida distribution space $(S^*)$. Then, a continuous function
$
\delta_Y (\cdot) : \mathbb{R}\longrightarrow (S)^*
$
is called Donsker delta function of $Y$ if it has the property that
$$
\int_{\mathbb{R}} g(y)\delta_Y(y)dy = g(Y)\ a.s.
$$
for all measurable functions $g : \mathbb{R} \longrightarrow \mathbb{R}$ such that the integral converges.
Herein the integral in the left is interpreted as a Bochner integral.
\end{definition}

For explicit representations of the Donsker delta function see, for instance, Chapter 7 in \cite{di2009malliavin}.

\section{Anticipative Portfolio Optimization (APO)}
\label{APO}

We formulate the insider trading problem in a Black-Scholes market with two assets, as anticipated in the Introduction. We suppose the insider has additional information about the underlying noise at the horizon time, specifically that the driving process takes the value $b \in \R$ at time $T$; note that this value is assumed to
be a constant rather than a random variable. In order to model the insider knowledge, we use the so called generalized Brownian bridge ending in $b$
(which is constant but otherwise arbitrary) as the driving process. More precisely, we consider the conditioned Gaussian process
\begin{equation}
(B_t \vert \ B_0=0,\ B_T=b),\ \ t\in [0,T],
\label{Bridge}
\end{equation}
which is characterized by its mean $b \ t/T$ for each time $t$ and its autocorrelation function $s(1-\frac{t}{T})$ between the temporal points $s$ and $t$ for $0 \le s \le t \le T$.

A simple representation of this Brownian bridge is given by (see \cite{rogers2000diffusions}, page 86)
\begin{equation}
\bar{B}_t = W_t - (W_{T}-b)\dfrac{t}{T},\ t\in [0,T],
\label{easyBB}
\end{equation}
where $W_t$ stands for the standard Brownian motion. The SDE for \eqref{easyBB} is
\begin{equation}
d\bar{B}_t = dW_t - \dfrac{W_{T}-b}{T} dt,\ t\in[0,T],  \ \bar{B}_0=0.
\label{easyBB_dB}
\end{equation}
We use this representation in the forward and Skorokhod schemes later on. Another representation of the generalized Brownian bridge is given by
(see \cite{klebaner2005}, page 132)
\begin{equation}
\widehat B_t = \int_0^{t}  \frac{T-t}{T-s} dW_s  +b \frac{t}{T}, t\in[0,T),
\label{integralBB}
\end{equation}
which we use for the first example of APO, and that satisfies the following SDE
\begin{equation}
d\widehat B_t = dW_t-\frac{\widehat B_t-b}{T-t}\ dt\ ,\ \ t\in[0,T),\ \ \widehat B_0=0.
\label{integralBB_dB}
\end{equation}
Note that this equation has all terms adapted, unlike equation~\eqref{easyBB_dB}; in particular, these equations give rise to two stochastic processes that share the same law, but cannot be compared at the pathwise level. Although at this time both SDEs can be interpreted samplewise, given the
additive nature of their noise, this fact will be crucial in the following sections. Indeed, it will allow us to compare the different notions
of anticipating stochastic calculus in the present financial context.

To optimize the portfolio for each integral, we first give the driving process, then solve the resulting SDE, afterwards compute the portfolio that maximizes the value of the problem, and finally compute this value with the detected optimal portfolio.
This, in particular, allows us to compare the different results that arise from the two notions of anticipating stochastic integration.
We start with the simplest case, which shows that modeling the driving stochastic process as a Brownian bridge reproduces the classical results on
the subject, see \cite{pikovsky1996anticipative}.

\subsection{APO with Brownian Bridge}
\label{APO_BB}
$ $

We start with an example of how the problem can be handled using the representation of the generalized Brownian bridge given in \eqref{integralBB}. In this case, the wealth process of the insider trader is modeled by
\begin{equation}
\begin{array}{ll}
dX_t &= (1-\pi_t)X_t\ r_t \ dt\ +\ \pi_tX_t (\mu_t\ dt +\sigma_td\widehat B_t ),
\\ \ X_0 &\in \mathbb{R}^+,
\end{array}
\label{dX_APO_BB}
\end{equation}

\noindent where we assume that $\mu_t, r_t, \sigma_t \in L^{\infty}([0,T])$, with $\sigma_t>0$, are deterministic and $\widehat B_t $ is given by \eqref{integralBB}.

\begin{theorem}\label{APO_BB_th}
Let $\pi_t \in L^{2}([0,T])$ be a deterministic function of time. Then, the optimal portfolio that maximizes $\mathbb{E}\left[\log (X_T/X_0) \right]$, where $X_t$ solves \eqref{dX_APO_BB}, is
\begin{equation}\nonumber
\pi_t^{*} = \dfrac{\mu_t - r_t}{\sigma_t^2} + \dfrac{b }{\sigma_t T}  ,\ \ \ t\in[0,T],
\end{equation}
and the corresponding value is
\begin{eqnarray}\nonumber
V_{T}^{\pi^*} &=& \displaystyle\int_0^{T} \left[ r_t + \dfrac{1}{2}  \left(  \dfrac{\mu_t-r_t}{\sigma_t} + \dfrac{b}{T} \right)^2 \  \right] dt.
\end{eqnarray}
\end{theorem}

\begin{proof}
From \eqref{integralBB} and \eqref{dX_APO_BB} we find

$$
\begin{array}{lll}
dX_t &= r_t(1-\pi_t)X_t\ dt\ +\ \pi_tX_t\ \left[\mu_t\ dt\ +\ \sigma_t\left(dW_t-\dfrac{\widehat B_t-b}{T-t}\ dt\right)\right]
\\\\&= r_t(1-\pi_t)X_t\ dt\ +\ \pi_tX_t\ \left[\mu_t\ dt\ -\ \sigma_t\ \dfrac{\widehat B_t-b}{T-t}\ dt + \sigma_tdW_t\right]
\\\\&= \left[r_t(1-\pi_t)\ +\ \pi_t\mu_t\ -\ \pi_t\sigma_t\dfrac{\widehat B_t-b}{T-t}\right]X_t\ dt\ +\ \pi_t\sigma_t X_t\ dW_t,
\end{array}
$$

\noindent which is an It\^o stochastic differential equation.
We may therefore use It\^o lemma for $\log X_t$ to obtain
\begin{equation}
d\log X_t = \left[r_t(1-\pi_t)\ +\ \pi_t\mu_t\ -\ \pi_t\sigma_t\frac{\widehat B_t-b}{T-t}-\frac{1}{2}\pi_t^2\sigma_t^2\right]dt\ +\ \pi_t\sigma_t\ dW_t.
\end{equation}

Taking the expectation of the integral form, we have that the value of the problem is given by
$$
\begin{array}{ll}
\mathbb{E}\left[\log (X_T/X_0) \right] &= \mathbb{E}\left[\displaystyle\int_0^T\left[r_t(1-\pi_t)\ +\ \pi_t\mu_t\ -\ \pi_t\sigma_t\frac{\widehat B_t-b}{T-t}-\frac{1}{2}\pi^2_t\sigma^2_t\right]dt \right.
\\\\& \ \  \ + \left. \displaystyle\int_0^T\pi_t\sigma_tdW_t \right]
\\\\&= \displaystyle\int_0^T\left[r_t(1-\pi_t)\ +\ \pi_t\mu_t\ +\ \pi_t\sigma_t\frac{b}{T}-\frac{1}{2}\pi^2_t\sigma^2_t\right]dt\  ,
\end{array}
$$
since $ \mathbb{E}\left[\displaystyle\int_0^T\pi_t\sigma_tdW_t \right]=0$ and $ \mathbb{E}\left[\dfrac{b - \widehat B_t}{T-t}\right] = \dfrac{b }{T}$. Then, we consider the maximization of the term
\begin{equation}
J_t^{\pi^{}} = r_t(1-\pi_t) + \pi_t\mu_t + \pi_t\sigma_t\ \dfrac{b }{T}-\frac{1}{2}\pi^2_t\sigma^2_t,
\end{equation}

The first derivative of $ J_t^{\pi}$ with respect to $\pi_t$ is 
\begin{equation}
J_t^{' \pi^{}} = \mu_t - r_t + \sigma_t\ \dfrac{b }{T}-\pi_t\sigma^2_t,
\end{equation}
and we equal the derivative to zero to find the optimal portfolio $\pi_t^{*} $ that maximizes the integrand
\begin{equation}
\mu_t - r_t + \sigma_t\ \dfrac{b }{T}-\pi_t^{*} \sigma^2_t=0.
\end{equation}

Then, the portfolio that maximizes the value of the problem is
\begin{equation}
\pi_t^{*} = \dfrac{\mu_t - r_t}{\sigma_t^2} + \dfrac{b }{\sigma_t T}  ,\ \ \ t\in[0,T].
\label{pi_APO_BB}
\end{equation}

To know the value of the problem, we compute
$$
\begin{array}{lllllll}
J_t^{\pi^*} &= r_t + (\mu_t-r_t)\pi_t^* + \pi_t^*\sigma_t \dfrac{b}{T} - \dfrac{1}{2} \pi_t^{*^2}\sigma_t^2
\\\\
&= r_t + (\mu_t-r_t) \Bigg( \dfrac{\mu_t-r_t}{\sigma_t^2} + \dfrac{b}{\sigma_t T} \Bigg) + \Bigg( \dfrac{\mu_t-r_t}{\sigma_t^2} + \dfrac{b}{\sigma_t T} \Bigg)\sigma_t\dfrac{b}{T}
\\\\
&- \dfrac{1}{2} \Bigg( \dfrac{\mu_t-r_t}{\sigma_t^2} + \dfrac{b}{\sigma_t T} \Bigg)^2 \sigma_t^2
\\\\
&= r_t + \dfrac{(\mu_t-r_t)^2}{\sigma_t^2} + \dfrac{(\mu_t-r_t)b}{\sigma_t T} + \dfrac{(\mu_t-r_t) b}{\sigma_t T} + \dfrac{b^2}{T^2}
\\\\
&- \dfrac{1}{2} \Bigg[ \dfrac{(\mu_t-r_t)^2}{\sigma_t^2} + 2 \dfrac{(\mu_t-r_t)b}{\sigma_t T} + \dfrac{b^2}{T^2} \Bigg]
\\\\
&= r_t + \Bigg( \dfrac{\mu_t-r_t}{\sigma_t} \Bigg)^2 + 2 \dfrac{(\mu_t-r_t)b}{\sigma_t T} + \dfrac{b^2}{T} - \dfrac{1}{2} \Bigg( \dfrac{\mu_t-r_t}{\sigma_t} \Bigg)^2 - \dfrac{(\mu_t-r_t)b}{\sigma_t T} - \dfrac{1}{2} \dfrac{b^2}{T^2}
\\\\
&= r_t + \dfrac{1}{2} \Bigg[ \dfrac{\mu_t-r_t}{\sigma_t} + \dfrac{b}{T} \Bigg]^2,\ \  \textup{for every}\ \ t\in[0,T].
\end{array}
$$
And the value of the problem is
\begin{equation}
V_{T}^{\pi^*} = \displaystyle\int_0^{T} \left[ r_t + \dfrac{1}{2}  \left(  \dfrac{\mu_t-r_t}{\sigma_t} + \dfrac{b}{T} \right)^2 \  \right] dt.
\label{value_APO_BB}
\end{equation}
Since all the parameters are deterministic, the statement follows.
\end{proof}

\begin{remark}
Note that these results are fully compatible with those in \cite{pikovsky1996anticipative}.
\end{remark}

\subsection{APO with the Russo-Vallois Forward Integration Method}
\label{APO_fw}
$ $

In this section, we present the portfolio optimization process with Russo-Vallois forward integration. First, we use deterministic portfolios and parameters
as in the previous section.

We now suppose that the driving process $B_t$ of $S_t$ in \eqref{risky_model} is given by $\bar{B_t } =W_t - (W_{T}-b) \frac{t}{T} $,  $t \in[0,T]$, as in (\ref{easyBB}),  and  $\mu_t, r_t, \sigma_t$, $\pi_t$ are deterministic, just like in the last case studied.
If these parameters are regular enough (what will not be necessarily assumed in the following),
then, by the self-financing property and Theorem 8.12 in \cite{di2009malliavin},
the wealth process of the insider trader is given (after substitution of the Brownian bridge as the forcing process) by the Russo-Vallois forward process
\begin{equation}
\begin{array}{ll}
d^-X_t&=r_t(1-\pi_t)X_t\ dt\ +\ \pi_tX_t\ [\mu_t\ dt\ +\ \sigma_td^-\bar{B}_t],\\
\ \ \ X_0 &\in \mathbb{R}^+,
\end{array}
\label{dX_APO_fw_det}
\end{equation}
where $d^-\bar{B}_t = d^-W_t-\frac{W_{T}-b}{T}\ dt , \ t\in[0,T]$. This is the model we will use for the portfolio optimization.

\begin{remark}
Note that \eqref{dX_APO_fw_det} is not an It\^o SDE due to its dependence on the future value of the Brownian motion, $W_T$, in the drift.
And thus the need to use Russo-Vallois forward integration in this model.
\end{remark}

\begin{theorem}\label{APO_fw_th}
Let $\pi_t \in L^{2}([0,T])$ be a deterministic function of time. Then the optimal portfolio that maximizes $\mathbb{E}\left[\log (X_T/X_0) \right]$, where $X_t$ solves \eqref{dX_APO_fw_det}, is
\begin{equation}\nonumber
\pi_t^{*} = \dfrac{\mu_t - r_t}{\sigma_t^2} + \dfrac{b }{\sigma_t T}  ,\ \ \ t\in[0,T],
\end{equation}
and the corresponding value is
\begin{equation}\nonumber
V_{T}^{\pi^*} = \displaystyle\int_0^{T} \left[ r_t + \dfrac{1}{2}  \left(  \dfrac{\mu_t-r_t}{\sigma_t} + \dfrac{b}{T} \right)^2 \  \right] dt.
\end{equation}
\end{theorem}

\begin{proof}
First, compute
$$
\begin{array}{llll}
dX_t &= r_t(1-\pi_t)X_t\ dt\ +\ \pi_tX_t\ \left[\mu_t\ dt\ +\ \sigma_t\left(d^-W_t-\frac{W_{T}-b}{T}\ dt\right)\right]
\\\\&= r_t(1-\pi_t)X_t\ dt\ +\ \pi_tX_t\ \left[\mu_t\ dt\ -\ \sigma_t\frac{W_{T}-b}{T}\ dt\ +\ \sigma_td^-W_t\right]
\\\\&= r_t(1-\pi_t)X_t\ dt\ +\ \pi_t\mu_tX_t\ dt\ -\ \pi_t\sigma_t\frac{W_{T}-b}{T}X_t\ dt\ +\ \pi_t\sigma_tX_t\ d^-W_t
\\\\&= \left[r_t(1-\pi_t)\ +\ \pi_t\mu_t\ -\ \pi_t\sigma_t\frac{W_{T}-b}{T}\right]X_t\ dt\ +\ \pi_t\sigma_tX_t\ d^-W_t.
\end{array}
$$

We apply It\^o formula for Russo-Vallois forward integrals (as we see in Theorem \ref{ito_lemma_fw}) to $\log X_t$ and find
\begin{equation}
d^-\log X_t = \left[r_t(1-\pi_t)+\pi_t\mu_t-\pi_t\sigma_t\frac{W_{T}-b}{T}-\frac{1}{2}\pi_t^2\sigma_t^2\right]dt+\pi_t\sigma_td^-W_t.
\end{equation}

Taking the expectation of the integral form, we have that the value of the problem is given by
$$
\begin{array}{ll}
\mathbb{E}\left[\log (X_T/X_0) \right] &= \mathbb{E}\left[\displaystyle\int_0^T\left[r_t(1-\pi_t)+\pi_t\mu_t-\pi_t\sigma_t \frac{W_{T}-b}{T} -\frac{1}{2}\pi^2_t\sigma^2_t\right]dt\right.
 +\left.\displaystyle\int_0^t\pi_t\sigma_td^-W_t \right]
\\\\
&= \displaystyle\int_0^T \left[r_t(1-\pi_t)+\pi_t\mu_t+\pi_t\sigma_t \frac{b}{T} -\frac{1}{2}\pi^2_t\sigma^2_t \right]dt,
\end{array}
$$
since $\sigma_t$ and $\pi_t$ are deterministic, and then
$$
\mathbb{E} \Bigg[ \int_0^T \pi_t\sigma_t d^-W_t \Bigg] = \mathbb{E} \Bigg[ \int_0^T \pi_t\sigma_t dW_t \Bigg] = 0.
$$

In this case, we have to maximize
\begin{equation}
J_t^{\pi^{}} = r_t(1-\pi_t) + \pi_t\mu_t + \pi_t\sigma_t\  \frac{b}{T} -\frac{1}{2}\pi^2_t\sigma^2_t.
\end{equation}
The value of $\pi_t, \ t\in[0,T]$, that maximizes $J_t^{\pi^{}}$ is
\begin{equation}
\pi_t^* = \dfrac{\mu_t-r_t}{\sigma_t^2}+\dfrac{b }{\sigma_t T}\ ,\ \ \ t\in[0,T].
\label{pi_APO_fw_deterministic}
\end{equation}
By direct substitution, we obtain the value of the problem, which is
\begin{equation}
V_{T}^{\pi^*}
= \mathbb{E}\displaystyle\int_0^{T} \left[ r_t + \dfrac{1}{2}  \left(  \dfrac{\mu_t-r_t}{\sigma_t} + \dfrac{b}{T} \right)^2 \  \right] dt,
\label{value_APO_fw_deterministic}
\end{equation}
and since the integrand is deterministic, the statement follows.
\end{proof}

\begin{remark}
Note that term $b/T$ in \eqref{value_APO_fw_deterministic} is equivalent to
$$
\mathbb{E} \left[ \dfrac{b-\bar B_t}{T-t} \right]= \dfrac{b- \mathbb{E}(\bar B_t)}{T-t}.
$$
This is a usual term we have to add to the ratio $ \dfrac{\mu_t-r_t}{\sigma_t}$, which represents the extra information we have at time $t$.
In the next section, allowing stochastic parameters, we will have in this term the current value of the driving process instead of its expected value.
\end{remark}

\begin{remark}
Note that both Theorems \ref{APO_BB_th} and \ref{APO_fw_th} lead to the same results. This highlights the consistency of both approaches.
\end{remark}


\noindent\textbf{APO with stochastic parameters}

Now we face the problem allowing $\sigma_t,\mu_t,r_t \in L^{\infty}([0,T] \times \Omega)$ and $\pi_t \in L^{2}([0,T] \times \Omega)$
be stochastic parameters. For that, we present a portfolio optimization that follows the procedure presented by
{\O}ksendal and R{\o}se in \cite{oksendal2017} (see also Chapter 8 in \cite{di2009malliavin}).
We note that the random variable $W_{T}$ obviously belongs to the Hida distribution space,
and it also has a Malliavin differentiable Donsker delta function (Proposition 7.2, \cite{di2009malliavin}).

We use an enlargement of filtration representing the insider information by
$$
\G:=\{\mathcal{G}_t:\mathcal{G}_t=\mathcal{F}_t\vee\sigma(W_{T}),\ \ t\in[0,T],\ \ T>0\},
$$
where $\mathcal{F}$ is the natural filtration (the filtration of an ordinary trader) and $\mathcal{F}_t\vee\sigma(W_{T}) := \sigma(\mathcal{F}_t  \cup \sigma(W_{T}))$, the smallest sigma algebra generated by $\mathcal{F}_t$ and $\sigma(W_{T})$. Inspired by the self-financing property,
we postulate the SDE
\begin{equation}
\begin{array}{ll}
d^-X_t&=r_t(1-\pi_t)X_t\ dt\ +\ \pi_tX_t\ [\mu_t\ dt\ +\ \sigma_td^-\bar{B}_t],\\
\ \ \ X_0 &\in \mathbb{R}^+,
\end{array}
\label{dX_APO_fw}
\end{equation}
as a model of the insider wealth process $X_t^\pi$,
where $\bar{B_t}$ is defined as in (\ref{easyBB}). Then, the solution of this SDE is
\begin{equation}
\begin{array}{ll}
\log (X_T/X_0) &= \displaystyle\int_0^t \left[ r_s(1-\pi_s)\ +\ \pi_s\mu_s\ -\ \pi_s\sigma_s\frac{W_{T}-b}{T}\ -\frac{1}{2}\pi_s^2\sigma_s^2\right]ds
\\\\
& \ \ \  + \displaystyle\int_0^t\pi_s\sigma_s\ d^-W_s.
\end{array}
\label{log_APO_fw}
\end{equation}

\begin{theorem}\label{APO_fw_th2}
Let $\pi_t \in L^{2}([0,T] \times \Omega)$ be $\mathcal{G}_t-$adapted and $\sigma_t,\mu_t,r_t \in L^{\infty}([0,T] \times \Omega)$
be $\mathcal{F}_t-$adapted. Moreover assume that $\sigma_ t \pi_t \in \mathbb{D}_0$ in the sense of Definition \ref{defdo}
and it is forward integrable in the strong sense.
Then, whenever it exists, the optimal portfolio that maximizes $\mathbb{E}\left[\log (X_T/X_0) \right]$, where $X_t$ solves \eqref{dX_APO_fw}, is
\begin{equation}\nonumber
\pi_t^* = \dfrac{\mu_t-r_t}{\sigma_t^2}+\dfrac{b-\bar B_t}{\sigma_t(T-t)}\ ,\ \ \ t\in[0,T),
\end{equation}
and the corresponding value is
\begin{equation}\nonumber
V_{T}^{\pi^*} = \mathbb{E}\displaystyle\int_0^{T} \left[ r_t + \dfrac{1}{2}  \left(  \dfrac{\mu_t-r_t}{\sigma_t} + \dfrac{b-\bar{B}_t}{T-t} \right)^2 \right] dt.
\end{equation}
\end{theorem}

\begin{proof}
Herein, we will be following the developments in~\cite{oksendal2017}.
To express the value of the problem, we use the Corollary \ref{corollary_forward} that relates the Russo-Vallois forward integral with the Malliavin derivative, along with the tower property, to find that
$$
\begin{array}{ll}
\mathbb{E} \log (X_T/X_0) &= \mathbb{E}\left[\displaystyle\int_0^T\left(r_t+(\mu_t-r_t)\pi_t - \sigma_t\frac{W_{T}-b}{T}\pi_t - \frac{1}{2}\pi^2_t\sigma^2_t + \sigma_tD_t\pi_t\right)dt\right]
\\\\ &= \mathbb{E}\left[\displaystyle\int_0^T \mathbb{E}\left[r_t+(\mu_t-r_t)\pi_t - \sigma_t\frac{W_{T}-b}{T}\pi_t - \frac{1}{2}\pi^2_t\sigma^2_t + \sigma_tD_t\pi_t\Big\vert{\mathcal{F}}_t\right]dt\right],
\end{array}
$$
where we have used the Fubini–Tonelli theorem. To proceed with the maximization with respect to $\pi_t$, we use the notation $\pi_t = f(t, Y)$, where $Y = W_T$. Then, we need to maximize
\begin{equation}
\begin{array}{ll}
J(f):=& \mathbb{E} \left[(\mu_t-r_t) f(t,Y) - \sigma_t \dfrac{Y-b}{T}f(t,Y) - \dfrac{1}{2}f^2(t,Y) \sigma^2_t \right.
\\
& \ \ \ \ + \left.\sigma_tD_t f(t,Y)\Big\vert{\mathcal{F}}_t\right].
\end{array}
\label{J_forward_malliavin}
\end{equation}
To that end, we follow~\cite{oksendal2017} to express $Y=W_{T}$ in terms of a Malliavin differentiable Donsker delta function $\delta_Y(y)$:
$$f(t,Y) = \int_\mathbb{R} f(t,y)\delta_Y(y)dy,$$
$$Y f(t,Y) = \int_\mathbb{R} y f(t,y)\delta_Y(y)dy,$$
$$f^2(t,Y) = \int_\mathbb{R} f^2(t,y)\delta_Y(y)dy,$$
$$D_s f(t,Y) = \int_\mathbb{R} f(t,y) D_s\delta_Y(y)dy.$$
We substitute these expressions in \eqref{J_forward_malliavin} to obtain
\begin{equation*}
\begin{array}{lllll}
J(f) = &\mathbb{E}\left[(\mu_t-r_t)\displaystyle \int_\mathbb{R} f(t,y)\delta_Y(y)dy - \sigma_t \int_\mathbb{R} \frac{y-b}{T} f(t,y)\delta_Y(y)dy \right.
\\\\ &- \left.\dfrac{1}{2}\sigma_t^2\displaystyle \int_\mathbb{R} f^2(t,y)\delta_Y(y)dy + \sigma_t \int_\mathbb{R} f(t,y)D_t\delta_Y(y)dy \Big\vert{\mathcal{F}}_t\right]
\\\\\ \ \ \ \ \ \ = & \displaystyle \int_\mathbb{R} \left\lbrace(\mu_t-r_t)f(t,y) \mathbb{E}\left[\delta_Y(y)\vert{\mathcal{F}}_t\right]
- \sigma_t \frac{y-b}{T} f(t,y) \mathbb{E}\left[\delta_Y(y)\vert{\mathcal{F}}_t\right]\right.
\\\\ & \left.\left.-\dfrac{1}{2}\sigma^2_tf^2(t,y)\mathbb{E}\left[\delta_Y(y)\vert{\mathcal{F}}_t\right] + \sigma_t f(t,y)\mathbb{E}\left[D_t\delta_Y(y)\vert{\mathcal{F}}_t\right]\right\rbrace dy\right.
\end{array}
\end{equation*}
\begin{equation*}
\begin{array}{lllll}
 \ \ \ \ = & \displaystyle\int_0^T\left\lbrace(\mu_t-r_t- \sigma_t \frac{y-b}{T}) f(t,y)\mathbb{E}\left[\delta_Y(y)\vert{\mathcal{F}}_t\right]\right.
\\\\ & \left.\left.-\dfrac{1}{2}\sigma^2_tf^2(t,y)\mathbb{E}\left[\delta_Y(y)\vert{\mathcal{F}}_t\right] + \sigma_t f(t,y)\mathbb{E}\left[D_t\delta_Y(y)\vert{\mathcal{F}}_t\right]\right\rbrace dy\right. ,
\end{array}
\end{equation*}
since the integral is of Bochner type.
To find the value $f^*(t,y)$ that maximizes $J(f)$, we write
$$
(\mu_t-r_t-\sigma_t \frac{y-b}{T})\mathbb{E}[\delta_Y(y)\vert{\mathcal{F}}_t] - \sigma_t^2 f^*(t,y)\mathbb{E}[\delta_Y(y)\vert{\mathcal{F}}_t] + \sigma_t \mathbb{E}[D_t\delta_Y(y)\vert{\mathcal{F}}_t] = 0,
$$
since $J$ is quadratic on $f$.
This implies that
$$
\begin{array}{ll}
f^*(t,y)&= \dfrac{(\mu_t-r_t-\sigma_t\frac{y-b}{T})\mathbb{E}[\delta_Y(y)\vert{\mathcal{F}}_t]+\sigma_t \mathbb{E}[D_t\delta_Y(y)\vert{\mathcal{F}}_t]}{\sigma^2_t \mathbb{E}[\delta_Y(y)\vert{\mathcal{F}}_t]}
\\\\ &= \dfrac{\mu_t-r_t}{\sigma^2_t} - \dfrac{y-b}{\sigma_t T} + \dfrac{\mathbb{E}[D_t\delta_Y(y)\vert{\mathcal{F}}_t]}{ \sigma_t \mathbb{E}[\delta_Y(y)\vert{\mathcal{F}}_t]},
\end{array}
$$
and therefore,
$$
\begin{array}{ll}
f^*(t,y)&= \dfrac{\mu_t-r_t}{\sigma^2_t} - \dfrac{y-b}{\sigma_t T} + \dfrac{y-W_t}{\sigma_t(T-t)},
\end{array}
$$
where we used that the quotient $\dfrac{\mathbb{E}[D_t\delta_Y(y)\vert{\mathcal{F}}_t]}{\mathbb{E}[\delta_Y(y)\vert{\mathcal{F}}_t]}$ equals to $\dfrac{y-W_t}{T-t}$ (see Section~2.2 in~\cite{draouil2015donsker}).

Then, the portfolio $\pi^*_t$ that maximizes $\mathbb{E}[\ln X^{\pi}(T)]$ is
$$
\begin{array}{ll}
\pi^*_t &= \dfrac{\mu_t-r_t}{\sigma_t^2} - \dfrac{W_{T}-b}{\sigma_t T} + \dfrac{W_{T}-W_t}{\sigma_t(T-t)}
\\\\ &= \dfrac{\mu_t-r_t}{\sigma_t^2} - \dfrac{T (W_{T}-b) -t(W_{T}-b) -T (W_{T}) +T W_t}{\sigma_t T (T-t)}
\\\\ &=
\dfrac{\mu_t-r_t}{\sigma_t^2} - \dfrac{T{\bar B_{t}-T b}}{\sigma_tT (T-t)},
\end{array}
$$
where we have used that $T\bar{B_t} = T W_t - t (W_{T}-b)$. Therefore,
\begin{equation}
\pi_t^* = \dfrac{\mu_t-r_t}{\sigma_t^2}+\dfrac{b-\bar B_t}{\sigma_t(T-t)}\ ,\ \ \ t\in[0,T).
\label{pi_APO_fw}
\end{equation}


To find the value of the portfolio, consider the equality
\[
J_t^{\pi^*} = \mathbb{E} \bigg[ r_t + (\mu_t - r_t)\pi_t^* + \sigma_t \pi_t^* \frac{b - W_T}{T} - \frac{1}{2} {\pi_t^*}^2 \sigma_t^2 + \sigma_t D_t \pi_t^* \bigg],
\]
to compute
$$V_{T}^{\pi^*} = \mathbb{E}\displaystyle\int_0^{T} \left( J_t^{\pi^*} \right) dt.$$
First, we write the portfolio in terms of the Wiener process
\[
\pi_t^* = \frac{\mu_t - r_t}{\sigma_t^2} + \frac{b - W_T}{\sigma_t T} + \frac{W_T - W_t}{\sigma_t (T - t)},
\]
to compute its Malliavin derivative
\[
D_t \pi_t^* = \frac{1}{\sigma_t} \left( \frac{1}{T - t} - \frac{1}{T} \right) = \frac{t}{\sigma_t T (T - t)}.
\]
Substituting these expressions in the argument of the expected value of $J_t^{\pi^*}$ we have that $ r_t + (\mu_t - r_t)\pi_t^* + \sigma_t \pi_t^* \frac{b - W_T}{T} - \frac{1}{2} {\pi_t^*}^2 \sigma_t^2 + \sigma_t D_t \pi_t^*$ equals to:
\begin{equation*}
\begin{array}{lll}
  & r_t + \left( \dfrac{\mu_t - r_t}{\sigma_t} \right)^2 + \dfrac{\mu_t - r_t}{\sigma_t} \dfrac{b - W_t}{T} + \dfrac{\mu_t - r_t}{\sigma_t} \dfrac{W_T - W_t}{T - t} \\\\
&+ \dfrac{\mu_t - r_t}{\sigma_t} \dfrac{b - W_t}{T} + \left( \dfrac{b - W_t}{T} \right)^2 + \dfrac{b - W_t}{T} \dfrac{W_T - W_t}{T - t} \\\\
&- \dfrac{1}{2} \left( \dfrac{\mu_t - r_t}{\sigma_t} \right)^2 - \dfrac{1}{2} \left( \dfrac{b - W_t}{T} \right)^2 - \dfrac{1}{2} \left( \dfrac{W_T - W_t}{T - t} \right)^2 \\\\
&- \dfrac{\mu_t - r_t}{\sigma_t} \dfrac{b - W_t}{T} - \dfrac{\mu_t - r_t}{\sigma_t} \dfrac{W_T - W_t}{T - t} - \dfrac{b - W_t}{T} \dfrac{W_T - W_t}{T - t} + \dfrac{t}{T(T - t)} \\\\
&= r_t + \dfrac{1}{2} \left( \dfrac{\mu_t - r_t}{\sigma_t} \right)^2 + \dfrac{\mu_t - r_t}{\sigma_t} \dfrac{b - W_t}{T} + \dfrac{1}{2} \left( \dfrac{b - W_t}{T} \right)^2 \\\\
&- \dfrac{1}{2} \left( \dfrac{W_T - W_t}{T - t} \right)^2 + \dfrac{t}{T(T - t)} \\\\
&= r_t + \dfrac{1}{2} \left( \dfrac{\mu_t - r_t}{\sigma_t} + \dfrac{b - W_t}{T} \right)^2 - \dfrac{1}{2} \left( \dfrac{W_T - W_t}{T - t} \right)^2 + \dfrac{t}{T(T - t)}.
\end{array}{}
\end{equation*}
Now, it is possible to calculate the value of $J_t^{\pi^*}$:
\[
\begin{array}{lll}
J_t^{\pi^*} & = r_t + \dfrac{1}{2} \mathbb{E} \left( \dfrac{\mu_t - r_t}{\sigma_t} - \dfrac{W_t}{T} + \dfrac{b}{T} \right)^2
- \dfrac{1}{2} \mathbb{E} \left( \dfrac{W_T - W_t}{T - t} \right)^2 + \dfrac{t}{T(T - t)} \\\\
& = r_t + \dfrac{1}{2} \mathbb{E} \left( \dfrac{\mu_t - r_t}{\sigma_t} - \dfrac{W_t}{T} \right)^2
+ \dfrac{1}{2} \left( \dfrac{b}{T}\right)^2 \\\\
& \ \ + \mathbb{E} \left( \dfrac{\mu_t - r_t}{\sigma_t} \right) \dfrac{b}{T} - \dfrac{1}{2(T-t)} + \dfrac{t}{T(T - t)} \\\\
& = r_t + \dfrac{1}{2} \mathbb{E} \left( \dfrac{\mu_t - r_t}{\sigma_t} - \dfrac{W_t}{T} + \dfrac{b}{T} + \dfrac{W_T - W_t}{T - t} \right)^2 \\\\
& = r_t + \dfrac{1}{2} \mathbb{E} \left( \dfrac{\mu_t - r_t}{\sigma_t} + \dfrac{b - b \, \frac{t}{T}}{T - t} + \dfrac{W_T \, \frac{t}{T} - W_t}{T - t} \right)^2 \\\\
& = r_t + \dfrac{1}{2} \mathbb{E} \left( \dfrac{\mu_t - r_t}{\sigma_t} + \dfrac{b - \bar{B_t}}{T - t} \right)^2.
\end{array}
\]
Hence, the value of the problem is
\begin{equation}
V_{T}^{\pi^*} = \mathbb{E}\displaystyle\int_0^{T} \left[ r_t + \dfrac{1}{2}  \left(  \dfrac{\mu_t-r_t}{\sigma_t} + \dfrac{b-\hat{B}_t}{T-t} \right)^2  \right] dt.
\label{value_APO_fw}
\end{equation}
\end{proof}

\begin{remark}
Theorem \ref{APO_fw_th2} recovers classical results of insider trading. However, it is apparently not consistent with Theorem \ref{APO_fw_th} in the sense that, if we assumed the deterministic character of the parameters, the present results do not reduce to the previous ones. This is not the consequence of a mistaken development: simply the assumptions are different. Precisely, the current portfolio process is anticipating (it depends on a future value of Brownian motion) and the former is a deterministic function. It is worth noting that the anticipative portfolio should achieve at least as good an outcome as the deterministic one, since the optimization problem under anticipation is less constrained.
\end{remark}

\begin{remark}
Note that, contrary to what happened previously, a divergence is present for $t=T$; this is the consequence of allowing stochastic portfolios.
\end{remark}

\subsection{APO with the Skorokhod Integration Method }
\label{APO_sk}
$ $

In this section, we present the portfolio optimization process using Skorokhod integration. To find a solution of the corresponding equation, we use anticipative Girsanov transformations, first, allowing the parameters to be stochastic and then, deterministic and finally constant to find a closed-form solution.


Again inspired by the self-financing property, we postulate the insider wealth is given by the process that solves
\begin{equation}
\begin{array}{ll}
\delta X_t = [\mu_t \pi_t + r_t(1-\pi_t)]X_td_t + \sigma_t\pi_tX_t\delta\bar{B}_t\\
\ \ X_0 \in \mathbb{R}^+,
\end{array}
\label{dX_MAPO_sk}
\end{equation}
where $$ \delta\bar{B}_t = \delta W_t -\dfrac{W_{T}-b}{T}dt,$$
$$\bar{B_t } = W_t - (W_{T}-b) \frac{t}{T} , \  t \in[0,T],$$
and $\delta$ denotes Skorokhod integration. Then
$$
\delta X_t  = \left[r_t(1-\pi_t)\ +\ \pi_t\mu_t\ -\ \pi_t\sigma_t\frac{W_{T}-b}{T}\right]X_t\ dt\ +\ \pi_t\sigma_tX_t\ \delta W_t.
$$
To meet the assumptions in section \ref{Anticipative} we need to substitute $W_T$ by $W_{T \wedge \tau}$, where $\tau$ is the stopping time $\tau = \inf\{t>0: |W_t| = m\sqrt{T} \}, \ m\in\N$, so it becomes a bounded random variable.
We do this in the hope that the limit $m \to \infty$ will yield a solution to problem \eqref{dX_MAPO_sk}; we will get back to this issue below.
For the model parameters, we use the assumptions in that section
so that its developments can be applied.

We use equation \eqref{sol_sk_stochastic} to find that the solution of \eqref{dX_MAPO_sk} is
\begin{equation}
X_t/X_0 = \exp\Bigg\{\int_0^t [\hat{\mu_s} \pi_s + r_s(1-\pi_s)] (U'_{s,t})ds \Bigg\}L'_t, \  t\in[0,T],
\label{sol_APO_sk_stochastic}
\end{equation}
with $L'_t$ as given in \eqref{Lt_stochastic}, where $\hat{\sigma_t} = \sigma_t \pi_t$ and $\hat{\mu_t} = (1-\pi_t)r_t + \pi_t \mu_t + \pi_t \sigma_t \dfrac{b-W_{T \wedge \tau}}{T}$.

When taking the expectation of the integral form, we have that the value of the problem is given by
$$
\begin{array}{ll}
\mathbb{E}[\log (X_T/X_0)] =& \mathbb{E} \displaystyle\int_0^T \Bigg[\mu_t \pi_t (U'_{t,T}) + r_t(1-\pi_t)(U'_{t,T})-\dfrac{1}{2}\sigma_t^2\pi_t^2(U'_{t,T})) \\\\
&-\dfrac{W_{T \wedge \tau}-b}{T}\sigma_t\pi_t(U'_{t,T}) \Bigg]dt
- \ \mathbb{E}\displaystyle\int_0^T \int_t^T (D_u\sigma_t\pi_t)(U'_{t,T})D_t[\sigma_u\pi_u(U'_{t,T})]du dt, \\
\end{array}
$$
where we used that $\mathbb{E}\displaystyle\int_0^T \sigma_t\pi_t(U'_{t,T})\delta W_t = 0$.

As $\sigma_t$ and $\pi_t$ are adapted to the filtration $\mathcal{F}_t$, then $D_u (\sigma_t \pi_t) = 0$ for $u > t$, so the expected value of the term
$$\int_0^T \int_t^T (D_u\sigma_t\pi_t)(U'_{t,T})D_t[\sigma_u\pi_u(U'_{t,T})] du dt$$
equals zero. 
Then, we apply the Fubini-Tonelli theorem and the tower property to compute
$$
\begin{array}{lll}
\mathbb{E}[\log (X_T/X_0)]
&= \mathbb{E}\Bigg[ \displaystyle\int_0^T \mathbb{E}\Bigg[\Big(\mu_t \pi_t (U'_{t,T}) + r_t(1-\pi_t)(U'_{t,T})-\dfrac{1}{2}\sigma_t^2\pi_t^2(U'_{t,T})
\\\\
&-\dfrac{W_{T \wedge \tau}-b}{T}\sigma_t\pi_t(U'_{t,T}) \Big\vert\mathcal{F}_t\Bigg]dt\Bigg]
\\\\
&= \mathbb{E}\Bigg[ \displaystyle\int_0^T \mathbb{E}\Bigg[\Big(\mu_t \pi_t + r_t(1-\pi_t)-\dfrac{1}{2}\sigma_t^2\pi_t^2
-\dfrac{W_{T \wedge \tau}(U'_{t,T})-b}{T}\sigma_t\pi_t \Big\vert\mathcal{F}_t\Bigg]dt\Bigg].
\end{array}\\
$$

From this step, the problem would be similar to the one formulated with the Russo-Vallois forward integration method, but with anticipative transformations, if the term $W_{T \wedge \tau}(U'_{t,T})$ had a closed expression; but it has not that in general. So, we need to further constraint the parameters to be constant, or at least deterministic, to achieve a closed-form solution.

\noindent\textbf{APO with deterministic parameters}

Now, we use that  $\mu_t, r_t, \sigma_t,\pi_t \in L^{\infty}([0,T])$ in the equation
\begin{equation}
\delta X_t = \Bigg[ (1-\pi_t)r_t + \pi_t\mu_t + \pi_t\sigma_t\dfrac{b-W_{T}}{T} \Bigg] X_tdt + \pi_t\sigma_t\ X_t\ \delta W_t, \ X_0 \in \mathbb{R}^+;
\label{dX_APO_sk_det}
\end{equation}
that is, we assume our parameters to be deterministic rather than stochastic. Note that, under this assumption and in those cases in which both solution
and parameters are regular enough, the equation follows as a consequence of the self-financing assumption and Theorem 8.20
in \cite{di2009malliavin}.

To solve the above equation, we use \eqref{sol_sk} taking $\hat{\sigma_t} = \sigma_t \pi_t$ and $\hat{\mu_t} = (1-\pi_t)r_t + \pi_t \mu_t + \pi_t \sigma_t \dfrac{b-W_{T}}{T}$. Since we need  $\hat{\mu_t}$ to be bounded, we consider the truncated version of the Brownian motion $W_{T\wedge \tau}$, where $\tau$ is the stopping time $\tau = \inf \{t>0: |W_t| = m\sqrt{T} \}, \ m\in\N$.
In this way, we restrict the Brownian motion to the state space $[-m\sqrt{T},m\sqrt{T}]$, where $m$ represents the number of standard deviations to be considered for the random variable $W_{t}, \ t \in [0,T]$. Note that this trick is the same employed in the previous section since the original
equation is not solvable with the methods of \cite{buckdahn1989}; but it becomes solvable after that substitution. Contrary to what happened before,
explicit solutions will become available now, and moreover the case $m \to \infty$ will become accessible to our analysis.

In the present case we find that $ U_{s,t}(W_{T \wedge \tau}) = W_{T \wedge \tau} - \displaystyle\int_0^{T \wedge \tau} \mathbb{I}_{[s,t]}(r) \pi\sigma \ du $, and we can use this fact to find a closed formula for the solution of \eqref{dX_APO_sk}, which is given by
\begin{equation}
\begin{array}{ll}
X_t^{(m)}/X_0 =& \exp\Bigg\{ \displaystyle\int_0^t \pi_s \sigma_s \delta W_s - \dfrac{1}{2} \int_0^t \pi_s^2\sigma_s^2ds
\\
&+ \displaystyle\int_0^t \Bigg[ (1-\pi_s)r_s + \pi_s\mu_s + \pi_s\sigma_s \dfrac{b-(W_{T \wedge \tau} - \int_s^t \pi_u\sigma_u du)}{T} \Bigg] ds \Bigg\}.
\end{array}
\label{sol_APO_sk_det0}
\end{equation}
The next result shows that this family, understood as a sequence in $m$, has a well-defined limiting behavior.

\begin{lemma}
Let $\mu_t, r_t, \sigma_t,\pi_t \in L^{\infty}([0,T])$ be deterministic parameters. Then, the pointwise limiting process $X_t:=\lim_{m \to \infty} X_t^{(m)}$ exists and has the explicit representation:
\begin{equation}
\begin{array}{ll}
\lim_{m \to \infty} X_t^{(m)} =& X_0 \exp\Bigg\{ \displaystyle\int_0^t \pi_s \sigma_s \delta W_s - \dfrac{1}{2} \int_0^t \pi_s^2\sigma_s^2ds
\\
&+ \displaystyle\int_0^t \Bigg[ (1-\pi_s)r_s + \pi_s\mu_s + \pi_s\sigma_s \dfrac{b-(W_{T} - \int_s^t \pi_u\sigma_u du)}{T} \Bigg] ds \Bigg\},
\end{array}
\label{sol_APO_sk_det}
\end{equation}
where the convergence takes place uniformly in $t$ almost surely.
\end{lemma}

\begin{proof}
Note that the expression of $X_t^{(m)}$ is stochastic only through its dependence on the Brownian motion. It appears twice: in the stochastic integral and in the integrand of the Lebesgue integral. The integrand of the stochastic integral is deterministic, therefore this integral actually reduces to a Paley-Wiener-Zygmund integral, which is just defined pathwise, see Chapter~4 in~\cite{evanssdes}. The dependence through the integrand of the Lebesgue integral is reduced to a single presence of $W_{T \wedge \tau}$. Therefore the process $X_t^{(m)}$ is simply defined pathwise for each $m \in \mathbb{N}$. Now fix a realization of Brownian motion; the dependence of this path of $X_t^{(m)}$ on $m$ is only through $W_{T \wedge \tau}$, which is in turn independent of time. Therefore, we can trivially commute the integration and limit to reduce the problem to find the limiting behavior of $W_{T \wedge \tau}$. Recalling the definition of $\tau$, by the almost sure continuity of Brownian motion it is clear that $\tau \to \infty$ as $m \to \infty$ with probability one, so $T \wedge \tau=T$ for each fixed $T$; then it follows that $\lim_{m \to \infty} W_{T \wedge \tau}=W_T$ almost surely. Since the limiting procedure is independent of the time variable, and the result holds for almost every path of $X_t^{(m)}$, then the statement follows.
\end{proof}

Such a good behavior makes $X_t$ a potential candidate to be the solution of
the original problem; indeed, the following result shows that it is the unique solution.

\begin{theorem}\label{euskorokhod}
Let $X_t$ be as defined above and let $\mu_t, r_t, \sigma_t,\pi_t \in L^{\infty}([0,T])$ be deterministic parameters. Then, the unique solution to
the linear Skorokhod stochastic differential equation \eqref{dX_APO_sk_det} is given by $X_t$.
\end{theorem}

\begin{proof}
First of all, note that the theory of \cite{buckdahn1989} cannot be directly applied since the drift of equation \eqref{dX_APO_sk_det} includes an unbounded
random variable (the Gaussian variable $W_T$). However, as already noted, the perturbed equation
\begin{equation}
\delta X_t = \Bigg[ (1-\pi_t)r_t + \pi_t\mu_t + \pi_t\sigma_t\dfrac{b-W_{T \wedge \tau}}{T} \Bigg] X_tdt + \pi_t\sigma_t\ X_t\ \delta W_t, \ X_0 \in \mathbb{R}^+,
\nonumber
\end{equation}
falls under the hypotheses of this theory for any fixed $m$, and therefore it follows that is possesses a unique solution, which is given by $X_t^{(m)}$.

On the other hand, if $\tau \ge T$, then $W_{T \wedge \tau} = W_{T}$, and the same result follows.
Now define
$$
\mathcal{M}:=\max_{0 \le t \le T} W_t, \qquad \mathfrak{m}:=\min_{0 \le t \le T} W_t;
$$
from Chapter 2 in \cite{ks1991} we know that
$$
\mathbb{P}\left(\left\{\mathcal{M} \ge m\sqrt{T}\right\}\right)= \sqrt{\frac{2}{\pi}} \int_{m}^{\infty} e^{-x^2/2}dx,
$$
and by symmetry
$$
\mathbb{P}\left(\left\{\mathfrak{m} \le -m\sqrt{T}\right\}\right)= \sqrt{\frac{2}{\pi}} \int_{-\infty}^{-m} e^{-x^2/2}dx.
$$
Therefore
\begin{eqnarray}
\nonumber
\mathbb{P}\left(\left\{W_{T \wedge \tau} = W_{T}\right\}\right) &=& \mathbb{P}\left(\left\{\mathcal{M} \le m\sqrt{T}\right\} \wedge \left\{\mathfrak{m} \ge -m\sqrt{T}\right\}\right) \\ \nonumber
&=& 1- \mathbb{P}\left(\left\{\mathcal{M} \ge m\sqrt{T}\right\} \vee \left\{\mathfrak{m} \le -m\sqrt{T}\right\}\right) \\ \nonumber
&\ge& 1- \mathbb{P}\left(\left\{\mathcal{M} \ge m\sqrt{T}\right\}\right) - \mathbb{P}\left(\left\{\mathfrak{m} \le -m\sqrt{T}\right\}\right) \\ \nonumber
&=& 1 - 2\sqrt{\frac{2}{\pi}} \int_{m}^{\infty} e^{-x^2/2}dx.
\end{eqnarray}
Consequently, for any fixed $m$, $X_t^{(m)}$ is the unique solution to equation \eqref{dX_APO_sk_det} with probability
$$
\mathbb{P}\left(\left\{X_t^{(m)} = X_t\right\}\right)=
\mathbb{P}\left(\left\{W_{T \wedge \tau} = W_{T}\right\}\right) \ge 1 - 2\sqrt{\frac{2}{\pi}} \int_{m}^{\infty} e^{-x^2/2}dx.
$$
Since $m$ is arbitrary, take the limit $m \to \infty$ to conclude that $X_t$ is the unique solution to \eqref{dX_APO_sk_det} almost surely.



\end{proof}

Once problem \eqref{dX_APO_sk_det} is solved, our aim is to compute the optimal portfolio for the expected logarithmic utility
\begin{eqnarray}\nonumber
V_T^{\pi} &:=& \mathbb{E}[\log(X_T^\pi/X_0^\pi)] \\ \nonumber
&=& \mathbb{E} \left\{ \int_{0}^{T} \left[ r_t + \left( \mu_t - r_t + \frac{\sigma_t \, b}{T} \right) \pi_t -\frac{\sigma_t^2}{2} \pi_t^2
+ \frac{\sigma_t}{T} \pi_t \left( \int_{t}^{T} \pi_s \, \sigma_s \, ds \right) \right] dt \right\} \\ \nonumber
&=& \int_{0}^{T} \left[ r_t + \left( \mu_t - r_t + \frac{\sigma_t \, b}{T} \right) \pi_t -\frac{\sigma_t^2}{2} \pi_t^2
+ \frac{\sigma_t}{T} \pi_t \left( \int_{t}^{T} \pi_s \, \sigma_s \, ds \right) \right] dt,
\end{eqnarray}
since all the parameters are deterministic.
For that, we will use the calculus of variations, which methods are legitimate under the current hypotheses with the additional assumption
of $\sigma_t>0$ (positivity of the volatility), see for instance Chapter~8 in~\cite{evanspdes}.
Following Section~8.1.2 in~\cite{evanspdes}, we define the real-valued function $\nu(\lambda):=V_T^{\pi}[\pi + \lambda \wp]$, where $\lambda \in \mathbb{R}$ and the function $\wp$ should be thought of as a perturbation of $\pi$. Now, we compute the first variation of $V_T^{\pi}$, which is the directional derivative (the Gateaux derivative) in the direction $\wp$, that is
$$
\frac{\delta V_T^{\pi}}{\delta \pi} := \left. \frac{d}{d \lambda} \nu(\lambda) \right|_{\lambda=0};
$$
it yields
\begin{eqnarray}\nonumber
\frac{\delta V_T^{\pi}}{\delta \pi} &=& \mathbb{E} \left\{ \int_{0}^{T} \left[ \mu_t - r_t +\frac{b}{T}\sigma_t -\sigma^2_t \pi_t + \frac{\sigma_t}{T}\int_{t}^{T}\pi_s \sigma_s \, ds
\right. \right. \\ \nonumber
& & \left. - \frac{\sigma_t}{T}\int_{t}^{T}\pi_s \sigma_s \, ds \right] \wp_t \, dt \\ \nonumber
& & \left. + \frac{1}{T} \left( \int_{0}^{T}\pi_s \sigma_s \, ds \right) \left( \int_{0}^{T}\wp_s \sigma_s \, ds \right) \right\}
\\ \nonumber
&=& \int_{0}^{T} \left[ \mu_t - r_t +\frac{b}{T}\sigma_t -\sigma^2_t \pi_t \right] \wp_t \, dt \\ \nonumber
& & + \frac{1}{T} \left( \int_{0}^{T}\pi_s \sigma_s \, ds \right) \left( \int_{0}^{T}\wp_s \sigma_s \, ds \right)
\\ \nonumber
&=& \int_{0}^{T} \left[ \mu_t - r_t +\frac{b}{T}\sigma_t -\sigma^2_t \pi_t
+ \frac{\sigma_t}{T} \left( \int_{0}^{T}\pi_s \sigma_s \, ds \right) \right] \wp_t \, dt,
\end{eqnarray}
after integration by parts.
Therefore, to find the critical points of $V_T^{\pi}$, we have to solve the equation
$$
\frac{\delta V_T^{\pi}}{\delta \pi}=0,
$$
which yields the Euler-Lagrange equation
$$
\mu_t - r_t +\frac{b}{T}\sigma_t -\sigma^2_t \pi_t
+ \frac{\sigma_t}{T} \left( \int_{0}^{T}\pi_s \sigma_s \, ds \right)=0.
$$
Dividing by $\sigma_t$ (which is positive by assumption) and rearranging terms leads to
$$
\sigma_t \pi_t = \frac{\mu_t - r_t}{\sigma_t} +\frac{b}{T} + \frac{1}{T} \left( \int_{0}^{T}\pi_s \sigma_s \, ds \right).
$$
By integrating this equality over $[0,T]$ we arrive at the compatibility condition
\begin{equation}\label{compatcond}
b= \int_{0}^{T} \frac{r_t -\mu_t}{\sigma_t} dt.
\end{equation}
The substitution of this value in the previous equation gives
$$
\sigma_t \pi_t - \frac{1}{T} \left( \int_{0}^{T} \sigma_s \pi_s \, ds \right) = \frac{\mu_t - r_t}{\sigma_t} -\frac{1}{T} \int_{0}^{T} \frac{\mu_t -r_t}{\sigma_t} dt;
$$
from this expression it is clear that the general solution of the integral equation is
\begin{equation}\label{optfamily}
\sigma_t \pi_t^* = \frac{\mu_t - r_t}{\sigma_t} + \mathcal{K} \Longleftrightarrow
\pi_t^* = \frac{\mu_t - r_t}{\sigma_t^2} + \frac{\mathcal{K}}{\sigma_t},
\end{equation}
for any arbitrary constant $\mathcal{K}$ (note that the integral equation follows from the compatibility condition~\eqref{compatcond}, but if this condition is not fulfilled, there are no critical points). The second variation (see Section~8.1.3 in~\cite{evanspdes}) is
$$
\frac{\delta^2{ V_T^\pi}}{\delta \pi^2} := \left. \frac{d^2}{d \lambda^2} \nu(\lambda) \right|_{\lambda=0}= -\int_{0}^{T} \sigma^2_t \, \wp^2_t \, dt + \frac{1}{T} \left( \int_{0}^{T} \sigma_s \, \wp_s \, ds \right)^2.
$$
Since the functional $V_T^\pi$ is quadratic on $\pi_t$, and by the Jensen inequality we have
$$
\left( \frac{1}{T} \int_{0}^{T} \sigma_s \, \wp_s \, ds \right)^2 \le \frac{1}{T} \int_{0}^{T} \left( \sigma_t \, \wp_t \right)^2 dt,
$$
where the inequality is strict except for constant integrands, so we deduce that the second variation is negative except through the direction $\wp_t \propto 1/\sigma_t$, along which it vanishes. Hence, we conclude that the one-parameter family of portfolios given in~\eqref{optfamily} yields the maximal logarithmic utility for all $\mathcal{K}$, and the resulting value is independent of this parameter. To find this value out, integrate by parts to get
\begin{eqnarray}\nonumber
V_T^{\pi} &=& \mathbb{E} \left\{ \int_{0}^{T} \left[ r_t + \left( \mu_t - r_t + \frac{\sigma_t \, b}{T} \right) \pi_t -\frac{\sigma_t^2}{2} \pi_t^2
\right] dt + \frac{1}{2T} \left( \int_{0}^{T} \pi_s \, \sigma_s \, ds \right)^2 \right\} \\ \nonumber
&=& \int_{0}^{T} \left[ r_t + \left( \mu_t - r_t + \frac{\sigma_t \, b}{T} \right) \pi_t -\frac{\sigma_t^2}{2} \pi_t^2
\right] dt + \frac{1}{2T} \left( \int_{0}^{T} \pi_s \, \sigma_s \, ds \right)^2.
\end{eqnarray}
Substituting $b$ by the integral expression in~\eqref{compatcond} and the portfolio by the optimal family
\begin{equation}
\begin{array}{lll}
\pi_t^* = \dfrac{\mu_t - r_t}{\sigma_t^2} + \dfrac{\mathcal{K}}{\sigma_t},
\end{array}
\label{pi_APO_sk_deterministic}
\end{equation}
yields for the value of the optimization problem
\begin{equation}
\begin{array}{lll}
V_T^{\pi_t^*} &= \displaystyle\int_{0}^{T} \! \left[ r_t + \dfrac12 \left( \dfrac{\mu_t - r_t}{\sigma_t} + \frac{b}{T} \right)^2 \right] dt \\ 
&= \displaystyle\int_{0}^{T} \left\{ r_t + \dfrac12 \left( \dfrac{\mu_t - r_t}{\sigma_t} + \dfrac{1}{T} \displaystyle\int_{0}^{T} \dfrac{r_s -\mu_s}{\sigma_s} ds \right)^2 \right\} dt,
\end{array}
\label{value_APO_sk_deterministic}
\end{equation}
which is, as anticipated, independent of the parameter $\mathcal{K}$.

The last functional recovers the result of the previous sections, in the case of deterministic (and bounded) parameters, but for the particular value of $b$ given by the compatibility condition~\eqref{compatcond}. Since we have to fix this value,
these developments have limited applicability; therefore, we will use different assumptions in the following, so we can go beyond the present result. Moreover, the upcoming calculations will show why the calculus of variations only allows access to the case characterized by this precise value of $b$, and why there is an infinite degeneracy of optimal portfolios in this case.


\noindent\textbf{APO with constant parameters}

To partially overcome the problem of limited applicability that we have just seen, from now on we assume that $\mu$, $r$, $\sigma$,
and $\pi$ are constants (i.e., time independent and deterministic). In this way, we plan to approach our problem without the constraint of fixing $b$ and also to improve the explicability of the previous results.

Therefore, we aim to solve the insider wealth equation
\begin{equation}
\begin{array}{lll}
\delta X_t &= [(1-\pi)r + \pi\mu]X_tdt + \pi\sigma X_t\delta\bar{B}_t\ ,\\
\  X_0 &\in \mathbb{R}^+,
\end{array}
\label{dX_APO_sk}
\end{equation}
where
$$\delta\bar{B}_t = \delta W_t + \dfrac{b-W_{T}}{T}dt,\ t\in[0,T].$$ Note that, in this particular case and prior to the substitution of the
Brownian bridge, the equation is not a consequence of the self-financing condition, but a consequence of the linearity of both the Skorokhod and Lebesgue integrals. Such a huge simplification comes from the assumption of the constancy of the parameters.

Substituting $\delta\bar{B}_t$ in~\eqref{dX_APO_sk}, we get
\begin{equation}\label{dX_APO_sk2}
\delta X_t = \Bigg[ (1-\pi)r + \pi\mu + \pi\sigma\dfrac{b-W_{T}}{T} \Bigg] X_tdt + \pi\sigma\ X_t\ \delta W_t.
\end{equation}
The solvability of this equation is guaranteed by Theorem~\ref{euskorokhod}. Note that, despite the apparent simplicity of equation~\eqref{dX_APO_sk2}, this case is not directly reducible to those treated with the It\^o or forward integrals. This is due to the presence of the non-adapted term $W_T$ in the drift of the equation, which makes the problem ill-posed in the sense of It\^o and genuinely different if considered as a forward stochastic differential equation. Hence, the specific treatment we carry out in this section.

Then, arguing as before, we find that the solution of~\eqref{dX_APO_sk2} is
\begin{equation}
\begin{array}{ll}
X_t/X_0 &= \exp\Bigg\{ \displaystyle\int_0^t \pi \sigma \delta W_s - \dfrac{1}{2} \int_0^t \pi^2\sigma^2ds
\\
& \quad + \displaystyle\int_0^t \Bigg[ (1-\pi)r + \pi\mu + \pi\sigma \dfrac{b-U_{s,t}(W_{T})}{T} \Bigg] ds \Bigg\}
\\\\
&= \exp\Bigg\{ \displaystyle\int_0^t \pi \sigma \delta W_s - \dfrac{1}{2} \int_0^t \pi^2\sigma^2ds
\\
& \quad + \displaystyle\int_0^t \Bigg[ (1-\pi)r + \pi\mu + \pi\sigma \dfrac{b-(W_T - \pi \sigma (t-s))}{T} \Bigg] ds \Bigg\},
\end{array}
\label{sol_APO_sk}
\end{equation}
since $ U_{s,t}(W_{T}) = W_{T} - \displaystyle\int_0^{T} \mathbb{I}_{[s,t]}(r) \pi\sigma \ du = W_T - \pi \sigma (t-s)$. Because of the constancy of the parameters, the integrals become straightforward, and this expression reduces to
\begin{equation}
X_t = X_0 \; \exp\Bigg\{ \pi \sigma W_t + \left[ (1-\pi)r + \pi\mu + \pi\sigma \frac{b-W_T}{T} - \dfrac{1}{2} \pi^2\sigma^2 \right] t + \dfrac{1}{2T} \pi^2\sigma^2 t^2 \Bigg\}
\end{equation}
for $0 \le t \le T$.
To compute the expected logarithmic utility at the horizon time $t=T$, first note that 
$$
X_T= X_0 \, \exp \left\{ \left[(1-\pi)r + \pi\mu \right] T + \pi\sigma b \right\},
$$
which is a deterministic quantity that resulted from the cancellation of both Brownian motions evaluated at the same time. Then, we can directly compute
$$
\begin{array}{rlrrrr}
V_T^{\pi}
:=& \mathbb{E}[\log(X_T^\pi/X_0^\pi)]
\\
=& \mathbb{E} \left\{ \left[(1-\pi)r + \pi\mu \right] T + \pi\sigma b \right\}
\\
=& r T + ( \mu - r)\pi T + \sigma b \pi.
\end{array}
$$
Clearly, as the expression for $V_T^{\pi}$ is affine in $\pi$, there exists neither a maximum nor a minimum. To address this issue, which was not encountered in the previous sections, we now introduce the no shorting condition. But prior to that, some comments are in order.

\begin{remark}
Note that there is exactly one exception to what we have just said. Indeed, if $b=(r-\mu)T/\sigma$, then $V^\pi_T=rT$ and the result is independent of $\pi$; in other words, we can take $\pi^* \in \mathbb{R}$ arbitrary. The situation is akin to that of the case with deterministic time-dependent parameters: maximization (which becomes trivial in this case of constant parameters) is only possible for the critical value of the fluctuation, and this case presents an infinite degeneracy of optimal portfolios. Now, by convention, we select $\pi^*=0$. This convention, in turn, fixes the value of $\mathcal{K}$ from the previous case, as we show below. 
\end{remark}

To find the value $\pi^*$ that maximizes $V_{T}(\pi)$, under no shorting, we consider the values of $b$ and the boundaries of $\pi$. We define $\theta := \dfrac{\mu - r}{\sigma}$. In consequence, we set,

if
$
b  \left\lbrace
\begin{array}{cclcl}
> - \theta T\ , & \textup{then} & \pi^* = 1 & \textup{and} & V_{T}( \pi^*) = \mu T + \sigma b, \\\
\le - \theta T\ , & \textup{then} & \pi^* = 0 & \textup{and} & V_{T}( \pi^*) = r T.
\end{array}
\right.\\
$

Therefore, the optimal portfolio under Skorokhod integration is
\begin{equation}
\pi^*=\mathbb{I}_{ \{b> - \theta T \} },
\label{pi_APO_sk}
\end{equation}
and the value of the problem in this case is
\begin{equation}
V_{T}( \pi^*) = r T + (\theta  \sigma T + \sigma b) \mathbb{I}_{ \{b> - \theta T \} }.
\label{value_APO_sk}
\end{equation}
Observe that the value of the problem in \eqref{value_APO_sk} is bounded by $r T$ and $ \mu T  + \sigma b$. This value, and the general result, are in deep contrast with all the results previously obtained in this work.

The strategy of the insider in this case consists in trading the risky asset if $b> - \theta T$ or the risk-free asset if $b \leq - \theta T$. In other words, the insider always chooses the winning bet. This does not only make full financial sense, but also deeply contrasts with our previous results in~\cite{bastons2018triple,elizalde2022chances,escudero2018,escudero2021optimal}, part of which point to the difficulty of using the Skorokhod integral in the present context. But the use of the logarithmic utility along with the no shorting condition sheds new light on this issue.

We finish this section with two remarks that allow to connect these results to those of the case with time-dependent, but deterministic, parameters.

\begin{remark}
The utility function~\eqref{value_APO_sk} takes the value $V_{T}( \pi^*) = r T$ when $b=-\theta T$. This corresponds, by the convention we have chosen, with $\pi^*=0$. Note that~\eqref{value_APO_sk_deterministic} generalizes~\eqref{value_APO_sk} for the analogous value of $b$ in the time-dependent case. Then, if we want a consistent convention in the case of~\eqref{pi_APO_sk_deterministic}, we need to choose
\begin{equation}
\mathcal{K} = \frac{1}{T} \int_0^T \dfrac{r_t -\mu_t}{\sigma_t} \, dt = \frac{b}{T},
\nonumber
\end{equation}
where $b$ in this formula is obviously given by~\eqref{compatcond}.
\end{remark}

\begin{remark}
The fact that maximization is impossible for $b \neq -\theta T$, since this makes the value function affine (with a non-vanishing slope), also has a reflection in the case of time-dependent parameters. If we choose, for example, the portfolio $\pi^\dagger=\mathcal{C}/\sigma_t$ with $\mathcal{C} \in \mathbb{R}$ arbitrary, and substitute it in the utility function, we find
\begin{eqnarray}\nonumber
V_T^{\pi} &=& \int_{0}^{T} \left[ r_t + \left( \mu_t - r_t + \frac{\sigma_t \, b}{T} \right) \pi^\dagger_t -\frac{\sigma_t^2}{2} \left(\pi^\dagger_t\right)^2
\right] dt + \frac{1}{2T} \left( \int_{0}^{T} \pi^\dagger_s \, \sigma_s \, ds \right)^2  \\ \nonumber
&=& \int_{0}^{T} \left[ r_t + \left( \frac{\mu_t - r_t}{\sigma_t} + \frac{b}{T} \right) \mathcal{C} \right] dt \\ \nonumber
&=& \int_{0}^{T} r_t \, dt + \left( \int_{0}^{T} \frac{\mu_t - r_t}{\sigma_t} \, dt + b \right) \mathcal{C}.
\end{eqnarray}
Therefore, if
$$
b \neq -\int_{0}^{T} \frac{\mu_t - r_t}{\sigma_t} \, dt,
$$
the function becomes unbounded in the limits $\mathcal{C} \to \pm \infty$, depending on the value of $b$. This implies that the no shorting condition needs to be applied to the case of time-dependent parameters as well, although the complete characterization of such a situation will be more involved.
\end{remark}

\subsection{Example of Performance}
\label{Algorithm}
$ $

So far, we have assumed that the insider has access to the future value of the Brownian motion driving the asset dynamics, rather than the terminal value of the asset itself. Although this may appear financially unnatural, since in practice insiders are more likely to have information about the asset price or other economic variables, such an assumption facilitates the use of anticipative stochastic calculus, for which the structure of the noise plays a central role, and its use is customary in the stochastic analytical literature. Importantly, when the asset follows a geometric Brownian motion with constant drift and volatility, knowing the future value of the Brownian motion is equivalent to knowing the future value of the asset: both are connected by a simple deterministic transformation. However, this equivalence might break down when these market parameters are not so simple, such as in the case of a time-dependent volatility. In such a situation, the relation between the asset price and the Brownian motion becomes more involved. Hence, we focus on the case of constant parameters to retain mathematical tractability, particularly because a closed-form solution for the Skorokhod strategy is only available under this assumption.

In order to demonstrate insider trading performance with the techniques described in this work, we exemplify the situation of a trader with privileged information. The features of the implementation are the following:

\begin{itemize}

\item Assumptions:\\
For ease of computation, we leave out the trading costs and the difference between the bid and ask prices, and we assume there is enough liquidity to trade.

\item Stock:\\
We use the 2-Year U.S. Treasury Note Future, a marketable risky instrument of the U.S. government traded in the Chicago Mercantile Exchange.

\item Parameters:\\
For the example we use constant parameters:
we compute $\sigma$ as the monthly standard deviation of the risk asset prices; we consider an average of the U.S. 3-Month Bond Yield to compute $r$; and we compute $ \mu$ from an ARIMA model of the log-return historic values $\log(S_t)-\log(S_{t-1})$.

\item Dates:\\
The trader starts to invest on March 03, 2019, and the horizon time is May 30, 2019. We assume the trader has privileged information about May 30.

\item Periodicity:\\
We consider daily prices at 14:00 (GMT-5). At that time, the trader computes the proportion of her wealth that should be in the risky asset ($\pi$) and the risk-free asset (1-$\pi$).

\end{itemize}

At time zero, the insider trader computes $b$ from the equation
$$\ S_{T} = S_0 \exp \{(\mu - \sigma^2/2)T+\sigma b\},$$
for the Russo-Vallois strategy and from 
$$\ S_{T} = S_0 \exp \{\mu T+\sigma b\}.$$
for the Skorokhod strategy.

It should be noticed that under this assumptions, is equivalent to know $S_T$ or $b$ since the market parameters, including the current value of the stock $S_0$, are known.

At time t (day t), the investor knows the value of $S_t$ and $r_t$. The wealth is computed as
\begin{equation}
X_t = X_{t-1} \exp\{(1-\pi_{t-1})r_{t-1}+\pi_{t-1}\log(S_t/S_{t-1})\},
\end{equation}
where:

\begin{itemize}
\item the portfolio of the Merton-honest strategy, \textit{i.e.}, without using insider information, is:
$\pi_t^{(ho)}=\dfrac{\mu_t-r_t}{\sigma^2_t}$.

\item the portfolio of the Russo-Vallois forward strategy is: $\pi^{(fw)}=\dfrac{\mu_t-r_t}{\sigma^2_t} + \dfrac{b}{\sigma_t T}$.

\item the portfolio of the Skorokhod strategy is: $\pi^{(sk)}=\mathbb{I}_{ \{b> - \theta T \} }$.
\end{itemize}

We show the wealth evolution $X_t$, $t \in [0,T]$ using three portfolios, the honest, the Russo-Vallois forward, and the Skorokhod portfolio in Figure~\ref{Fig.APO_TU}. We see that the wealth using the Skorokhod portfolio is bigger than using the Russo-Vallois portfolio. The wealth of the honest trader is far less than the previous ones not only at the end but practically in the whole period. It should be mentioned that this is a single-case example. In the following section we perform simulations with 50 thousand paths.

\begin{figure}[h]
\begin{center}
\includegraphics[
trim = 0cm 0cm 0cm 0cm, clip, 
width=12cm]{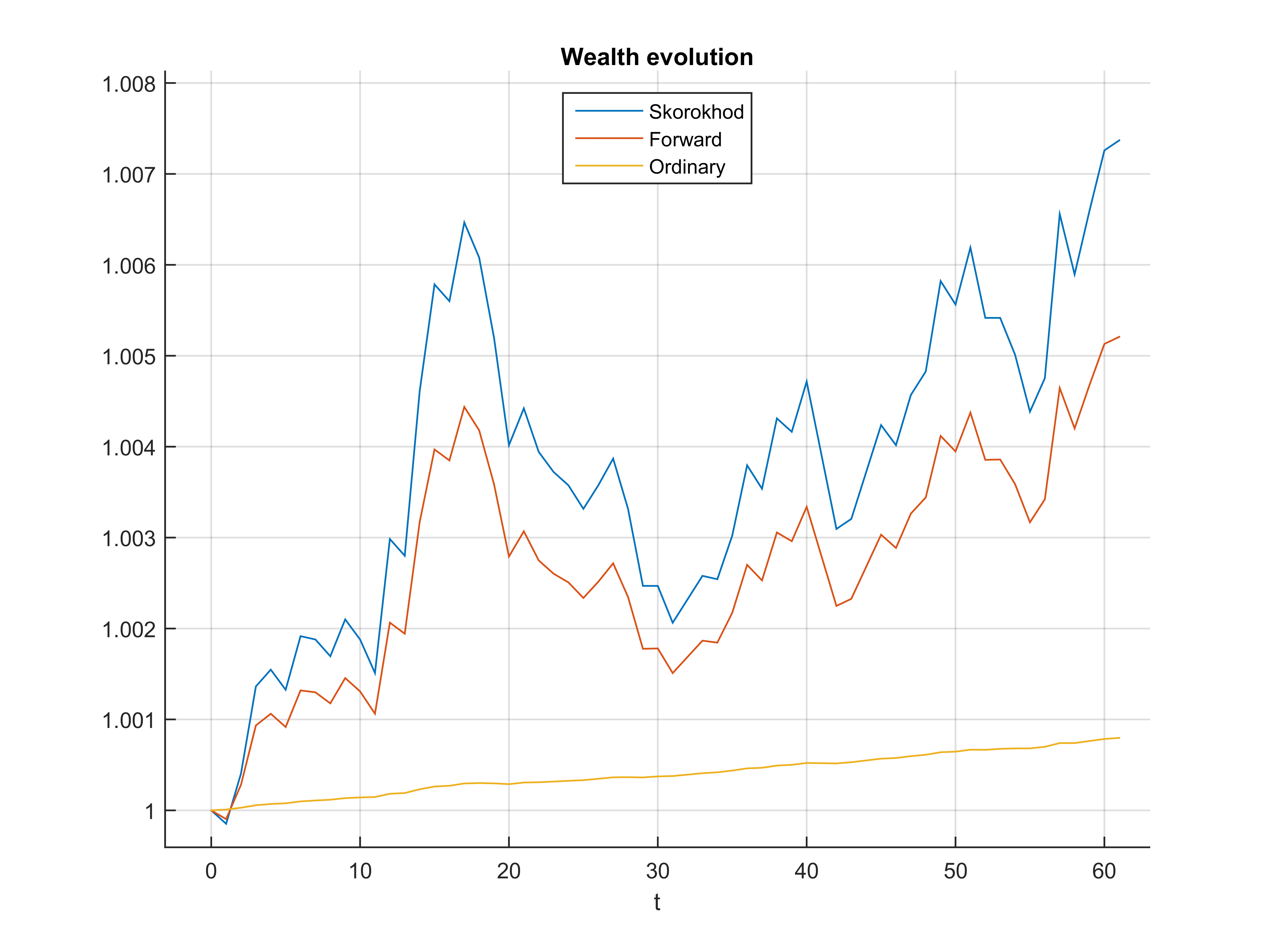}
\end{center}
\caption{Wealth evolution of the honest trader in yellow, the forward trader in red, and the Skorokhod trader in blue with the stock 2-Year U.S. Treasury Note Future.}
\label{Fig.APO_TU}
\end{figure}

As a final note, let us remark that the case with time-dependent parameters might give a different output. Although the portfolios corresponding to the first two strategies are formally time-dependent, we have assumed constant parameters, so in practice they are constant. The time-dependent case might need, in principle, a separate treatment, although this is complicated by the lack of explicit solutions in the Skorokhod case.

\subsection{Simulation}
$ $

In this section, we show how to perform a simulation of portfolio optimization from the point of view of both honest and insider trading. We consider two insider portfolios constructed with the Russo-Vallois forward integration approach and the Skorokhod integration one.

First, we simulate realizations of a conditioned Gaussian process
$$
(B_t \vert \ B_0=0,\ B_T=b),\ \ t\in [0,T].
$$
We start from the given extreme points $B_0$ and $B_T$. Then, recursively, given two values $B(u)$ and $B(t)$, we simulate the value $B(s)$ for $0<u<s<t<T$ as in \cite{glasserman2003monte}. The random vector
$[B(u)B(s)B(t)]^T$
is Gaussian with mean vector and covariance matrix:
$$
\begin{bmatrix}
B(u)\\
B(s)\\
B(t)
\end{bmatrix}
\sim N
\begin{pmatrix}
\begin{bmatrix}
0\\
0\\
0
\end{bmatrix}
,
\begin{bmatrix}
u & u & u\\
u & s & s\\
u & s & t
\end{bmatrix}
\end{pmatrix}
.
$$
Thus, the conditional distribution $(B(s)\mid B(u), B(t))$ is given by
$$
N \left( \dfrac{(t-s)B(u) + (s-u)B(t)}{t-u} , \dfrac{(s-u)(t-s)}{t-u} \right),
$$
and we can simulate $B(s)$ through the expression
$$
B(s) = \dfrac{(t-s)B(u) + (s-u)B(t)}{t-u} + \sqrt{\dfrac{(s-u)(t-s)}{t-u}}Z,
$$
where $Z\sim N(0,1)$.

In Figure~\ref{Fig.BB0} we illustrate multiple realizations of the algorithm that generates a Brownian bridge by conditioning it to start at a given point and return to a specified endpoint, in this case zero, a the fixed time horizon. This construction is particularly useful in modeling constrained random evolutions, where the terminal condition is known in advance.

\begin{figure}[h]
\begin{center}
\includegraphics[
trim = 0cm 0cm 0cm 0cm, clip, 
width=12cm]{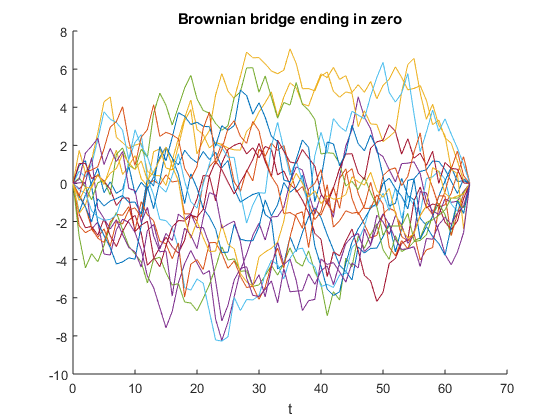}
\end{center}
\caption{Brownian bridges ending in zero.}
\label{Fig.BB0}
\end{figure}

We apply the simulation with different values of $b\sim N(e,64)$, so that each simulated path spans 64 trading days, which aligns with the example applied to the Treasury Bond Future in the previous section. In this simulation, we assume the financial instrument is represented by a geometric Brownian motion with initial value 100, $\mu=0.03$ and $\sigma=0.3$, and we consider a risk-free rate of $0.0027$ to compute the portfolio. For each path, we perform the algorithm described in Section \ref{Algorithm} which involve computing the value of a stochastic optimization problem under two different stochastic calculus frameworks: Russo Vallois forward and Skorokhod integration. Once the values under each integration approach are computed for every path, we repeat the entire simulation 50 thousand times to build distributions of terminal outcomes. This allows us to compare how the choice of integration strategy affects the overall distribution of the problem's solution.

In Figure~\ref{Fig.Distribution}, we compare different mean values of the distribution of $b$. At the top we use the expected value $e=1$, at the middle $e=0$, and at the bottom $e=-1$. We see that in all cases the mean value of the problem increases if the mean of $b$ increases and that, under Skorokhod strategy, the distribution of the value of the problem has bigger means than under Russo-Vallois forward strategy and the difference of variances are negligible.
Despite this difference, the resulting histograms display a remarkably similar overall shape, which suggests tat both approaches reflect a comparable underlying dynamic, even though their formulation differ.

\begin{figure}[ht]
\begin{subfigure}{}
  \centering
  \includegraphics[scale=.43]
  {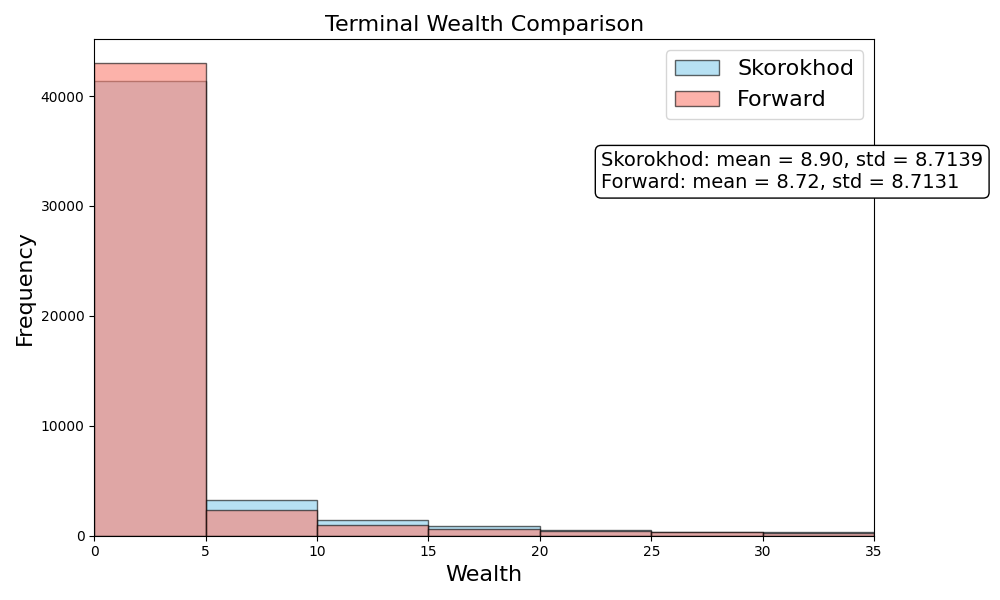}
  \label{jojojo}
\end{subfigure}
\begin{subfigure}{}
  \centering
  \includegraphics[scale=.43]
  {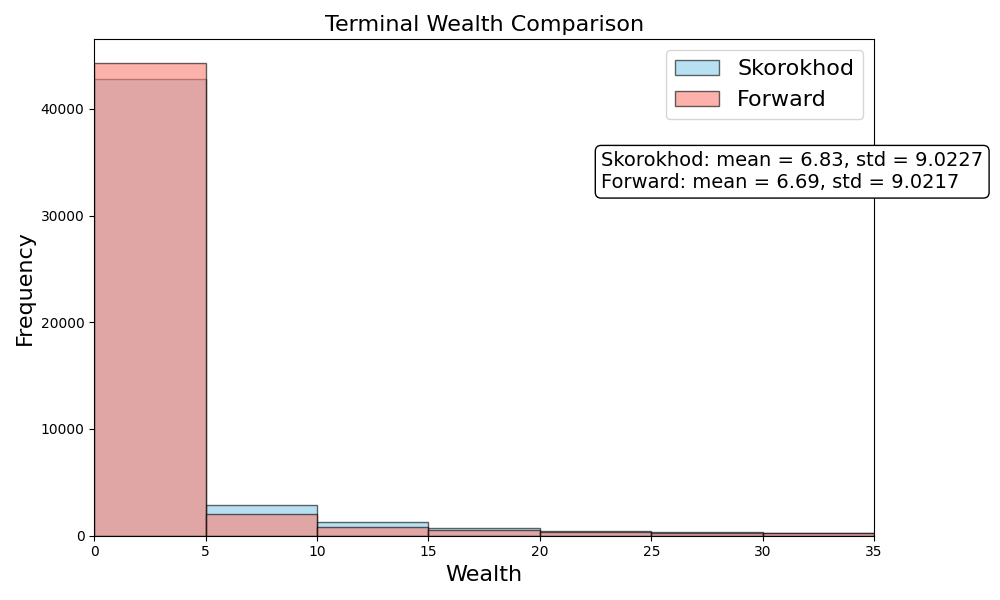}
\end{subfigure}
\begin{subfigure}{}
  \centering
  \includegraphics[scale=.43]
  {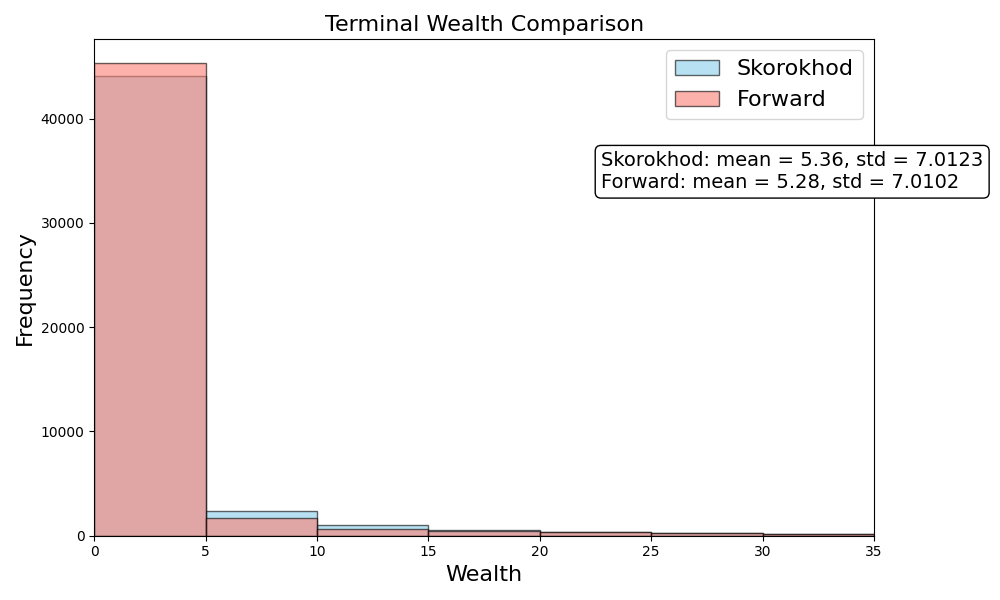}
\end{subfigure}

\caption{Histogram of $V_{T}^{\pi^*}(b)$ for different samples of $b$ under Russo Vallois forward and Skorokhod strategies, where $b\sim N(e,64)$. At the top we use the expected value $e=1$, at the middle $e=0$, and at the bottom $e=-1$.}
\label{Fig.Distribution}
\end{figure}

We remind the reader that we are employing constant parameters in our comparisons primarily due to a key limitation of the Skorokhod strategy: it admits a closed-form solution only under the assumption of constant coefficients. While this constraint does not apply to the other methods considered, such as the Russo-Vallois strategy, which could potentially yield better performance under time-varying parameters, we restrict the setting to maintain a fair and analytically tractable comparison, as we show in the next section.

\subsection{Properties and Comparison of Methods}
\label{Distribution}
$ $

Now we aim to describe the properties and compare the methods we have used through this manuscript for optimizing the insider portfolio.
We start with the case where the final value of the driving Brownian process, $b$, is constant. To obtain closed-form expressions, we consider the model parameters to be constant, too. 

In the previous sections, we obtained explicit expressions of optimal portfolios and values of the problem for which shorting was allowed, except in the final case (for which no-shorting permitted to interpret optimality). Therefore, in order to obtain a full comparison of the methods, we need to impose the no-shorting condition to all. We now show in detail how these expressions change when this condition is imposed.

The optimal portfolio for an honest trader is
$$
\pi^{(ho)} = \left( \left( \dfrac{\mu - r}{\sigma^2} \right) \wedge 1 \right) \vee 0.
$$
Therefore the value of the problem for this trader is
\begin{equation}\nonumber
V_T^{(ho)}
= r T + \dfrac{1}{2}\ \theta^2 T \ \mathbb{I}_{\big\{ \theta\in(0,\sigma) \big\}} + \Bigg(\theta\sigma - \frac{1}{2}\ \sigma^2\Bigg) T \ \mathbb{I}_{\big\{ \theta \geqslant \sigma \big\}}
= \left \{ \begin{matrix} r T & , & \theta \leqslant 0,
\\
r T + \frac{1}{2} \theta^2 T & , & 0 < \theta < \sigma,
\\
\mu T - \frac{1}{2}\sigma^2 T & , & \theta \geqslant \sigma,
\end{matrix}\right.
\end{equation}
where $\theta = \dfrac{\mu - r}{\sigma}$.

For the forward scheme, the optimal portfolio with constant parameters is
\begin{equation}
\begin{array}{ll}
\pi^{(fw)} &= \left( \left( \dfrac{\mu - r}{\sigma^2} + \dfrac{b}{T \sigma} \right) \wedge 1 \right) \vee 0
\\\\
& = \dfrac{\theta + \alpha}{\sigma} \mathbb{I}_{\left\lbrace \frac{\theta + \alpha}{\sigma} \in (0,1) \right\rbrace} + \mathbb{I}_{\left\lbrace \frac{\theta + \alpha}{\sigma} \geqslant 1 \right\rbrace}
\\\\
& = \dfrac{\theta + \alpha}{\sigma} \mathbb{I}_{\left\lbrace b \in (-\theta T , -\theta T + \sigma T) \right\rbrace} + \mathbb{I}_{\left\lbrace b \geqslant -\theta T + \sigma T \right\rbrace},
\end{array}
\label{FwdPiConst}
\end{equation}
where $\alpha = \dfrac{b}{T}$. And the value of the problem in terms of $\pi^{(fw)}$ is
$$
\begin{array}{ll}
V_T^{(fw)} &= \mathbb{E} \left[ \displaystyle\int_0^T \left( r + \theta \pi^* \sigma + \pi^* \sigma \dfrac{b}{T} - \dfrac{1}{2} \left( \pi^* \sigma \right)^2 \right) dt \right]
\\\\
& = r T + \theta \pi^* \sigma T + \pi^* \sigma b - \dfrac{1}{2} (\pi^* \sigma)^2 T.
\end{array}
$$

Substituting the result \eqref{FwdPiConst}, we have that
$$
V_T^{(fw)}
= \left\lbrace
\begin{array}{cll}
r T & , & b \leqslant - \theta T,
\\
r T + \dfrac{1}{2} (\theta + \alpha)^2 T & , & b \in (-\theta T , -\theta T + \sigma T],
\\
\mu T + \sigma b - \dfrac{1}{2} \sigma^2 T & , & b > -\theta T + \sigma T ,
\end{array}
\right.
$$
or, equivalently, using indicator functions
$$
\begin{array}{ll}
V_T^{(fw)} &= r T + \dfrac{1}{2} \left( \theta + \dfrac{b}{T} \right)^2 T\ \mathbb{I}_{\left\lbrace b \in (-\theta T , -\theta T + \sigma T] \right\rbrace} + \left( \theta \sigma T + \sigma b - \dfrac{1}{2} \sigma^2 T \right) \mathbb{I}_{\left\lbrace b > -\theta T + \sigma T \right\rbrace}.
\end{array}
$$

For the Skorokhod scheme, we already bounded the optimal portfolio:
$$
\pi^* = \mathbb{I}_{\{b>-\theta T\}},
$$
and recall that the value of the problem is
$$
V_{T}(\pi^*) = r T + (\theta\sigma T + \sigma b)\mathbb{I}_{\{ b>-\theta T \}}.
$$

Observe that if $b \leq -\theta T$, then $V_{T}^{(sk)} = V_{T}^{(fw)} = r T$, and in fact it is better to invest in the risk-free asset, given that we assume
$\theta>0$, which is a financially meaningful condition. Let us discuss the case $b >- \theta T$. Under this assumption, $\mu T  + \sigma b$ is bigger than $r T + \dfrac{1}{2} \left( \theta + \dfrac{b}{T} \right)^2 T$ if $b \le -\theta T + \sigma T$, then $V_{T}^{(sk)} > V_{T}^{(fw)}$, $ b \in (-\theta T , -\theta T + \sigma T]$. Finally, for the case $ b > -\theta T + \sigma T$, we also have that $V_{T}^{(sk)} > V_{T}^{(fw)}$ since $\mu T  + \sigma b$ is bigger than $\mu T + \sigma b - \dfrac{1}{2} \sigma^2 T$. Therefore, we conclude that the method involving Skorokhod integration yields a portfolio that is equally or more profitable for every value of $b$. The reader should observe that we are comparing the effect of the choice of integration rule on the final wealth obtained by each strategy. Since different integration rules lead to distinct price processes, the corresponding optimal strategies and wealth computations are inherently different. Each strategy optimizes under its own model dynamics, and thus the resulting wealth reflects the outcome of that specific modeling framework.

As an example, we perform a numerical comparison written in Matlab software of $V_T^{\pi^*}$ under Skorokhod and Russo-Vallois forward integration with the market parameters $\mu = .03$, $r = .02$, $\sigma = .30$, and $T=1$ to simplify the computations.

In Figure~\ref{Fig.Value}, we show $V_{T}^{\pi^*}$ with respect to $b$ in the interval $[{-\theta T}, -\theta T + \sigma T]$ under Skorokhod (blue line) and Russo-Vallois forward integration (red line). We also represent the investment of an honest trader (yellow line) without anticipating information, which value is constant with respect to $b$, and the safe investment (purple line), under the risk-free rate, which is also constant with respect to $b$.

\begin{figure}[ht]
\begin{subfigure}{}
  \centering
  \includegraphics[scale=.65]
  {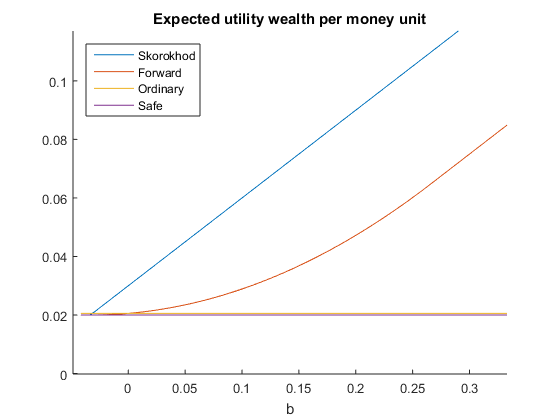}
  \label{fig:sub-first}
\end{subfigure}
\begin{subfigure}{}
  \centering
  \includegraphics[scale=.65]
  {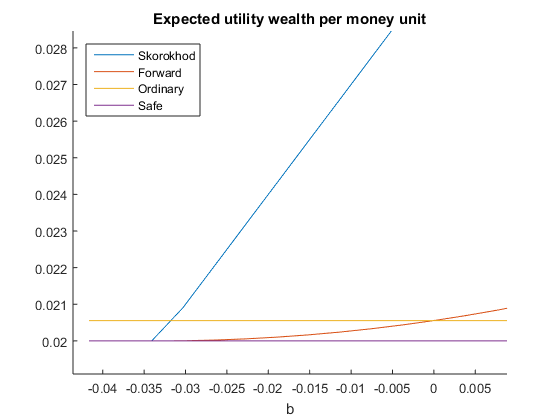}
  \label{fig:sub-second}
\end{subfigure}
\caption{$V_{T}^{\pi^*}(b)$ from $-\theta T$ to $-\theta T + \sigma T$ at the top and for negative values of $b$ at the bottom.}
\label{Fig.Value}
\end{figure}

So far, we have been assuming the constancy of $b$. If we considered $b$ as a Gaussian random variable (herein assumed to be independent of the Brownian motion), the value of the problem would become a random variable depending on $b$. In this sense, we can interpret the previous results as the conditional expectation of the insider utility given some fixed $b=\bar{b}$: $V_{T}^{\pi^*}(\bar{b}) = \mathbb{E}( \log (X_T/X_0) | b= \bar{b})$. Then, if we no longer
assume a fixed value of $b$, we can obtain the unconditional expectation by integrating the conditional expectation over the domain of $b$, due to its independence
of Brownian motion.

We have performed the computation of the unconditional expectation numerically, which is the value of the problem taking into account all the possible values of $b$.
In Figure~\ref{Fig.Normal}, we plot $V_{T}^{\pi^*}(b) \mathbb{P}(b)$ to visualize the area under this curve, that represents the integral for the unconditional expectation $V_{T}^{\pi^*}$ . We see that the Skorokhod curve of $V_{T}^{\pi^*}(b) \mathbb{P}(b)$ is above the Russo Vallois forward one, and therefore, the integral of the value of the problem is bigger under the Skorokhod scheme. We have assumed that $b \sim N(0,T)$ and have taken $T=1$.

\begin{figure}[ht]
\begin{subfigure}{}
  \centering
  \includegraphics[scale=.8]
  {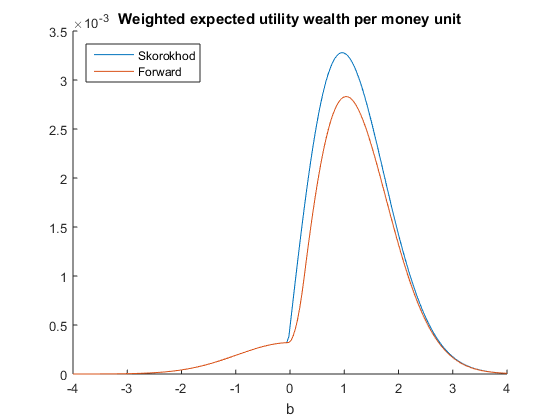}
  \label{fig:sub-first}
\end{subfigure}
\caption{$V_{T}^{\pi^*}(b) \mathbb{P}(b)$, with $b\sim \textup{N}(0,1)$.}
\label{Fig.Normal}
\end{figure}

Finally, we show the expressions of the unconditional expectations under forward and Skorokhod integration and under the assumption that $b$ is a Gaussian random variable. For the forward scheme, we find that
$$
\begin{array}{ll}
\mathbb{E} \left[ V_1^{(fw)} \right]
&= r T + \dfrac{1}{2} \mathbb{E} \left[ \left( \theta + b \right)^2 \mathbb{I}_{\left\lbrace b \in (-\theta , -\theta + \sigma ] \right\rbrace} \right]
\\\\
&  \ \ \ + \mathbb{E} \left[ \left( \theta \sigma  - \dfrac{1}{2} \sigma^2 + \sigma b \right) \mathbb{I}_{\left\lbrace b > -\theta + \sigma \right\rbrace} \right]
\end{array}
$$

$$
\begin{array}{ll}
 &= r + \dfrac{1}{2} \dfrac{1}{\sqrt{2\pi}} \displaystyle\int_{-\theta T}^{-\theta  + \sigma} \left( \theta + b \right)^2 e^{-b^2/2} db
\\\\
& \ \ \ + \dfrac{1}{\sqrt{2\pi}} \displaystyle\int_{-\theta + \sigma T}^{\infty} \left( \theta \sigma - \dfrac{1}{2} \sigma^2 + \sigma b \right) e^{-b^2/2} db
\\\\
&= \dfrac{1}{4} (\theta + 1) \text{erf} \left( \dfrac{\sigma - \theta}{\sqrt{2}} \right) + \dfrac{1}{4} (\theta + 1) \text{erf} \left( \dfrac{\theta}{\sqrt{2}} \right)
\\\\
& \ \ \ + \dfrac{1}{4} \sqrt{\dfrac{2}{\pi}}  \exp \left\{ -\dfrac{1}{2} (\theta^2 + \sigma^2) \right\} \left( (\theta - \sigma) \exp \left\{ \theta \sigma  \right\} - \theta \exp \left\{ \dfrac{\sigma^2}{2} \right\} \right)
\\\\
& \ \ \ + \dfrac{1}{4} \sigma (2 \theta - \sigma) \left( \text{erf} \left( \dfrac{\theta - \sigma}{\sqrt{2}} \right) + 1 \right) + \dfrac{ \sigma
}{\sqrt{2 \pi}}\exp \left\{  -\dfrac{1}{2} (\theta - \sigma)^2) \right\}.
\end{array}
$$

Under the Skorokhod scheme, we find that
$$
\begin{array}{ll}
\mathbb{E}\left( V_{1}^{(sk)} \right) &= r + \mathbb{E} \left[ (\theta\sigma   + \sigma b) \mathbb{I}_{\{b > -\theta\}} \right]
\\\\
&= r T + \dfrac{1}{\sqrt{2\pi}} \displaystyle\int^\infty_{-\theta } (\theta\sigma + \sigma b) e^{-b^2/2} db
\\\\
&= r  + \dfrac{\theta\sigma }{2}\left[ \text{erf} \left( \dfrac{\theta }{\sqrt{2}} \right) + 1 \right] + \dfrac{\sigma}{\sqrt{2\pi}} e^{-\theta^2 /2}.
\end{array}
$$

Note that these results do not reproduce the classical ones. The reason is that, under the present assumptions, the random variable $b$ is independent
of the Brownian motion, contrary to what has been classically assumed.

\section{Conclusions}

In this work we have studied the role of different notions of anticipating calculus on the maximization of the logarithmic utility of an insider trader.
It thus complements the previous studies in which this role was examined for risk-neutral traders. In \cite{escudero2018}, \cite{bastons2018triple}, and \cite{escudero2021optimal} it was shown that the Russo-Vallois forward integral produces intuitive results from the financial viewpoint, while the Skorokhod integral does not,
in the sense that it effectively transforms the insider trader into an uninformed one in terms of performance. In particular, in all these works, the Skorokhod integral provides
the insider with a wealth that is smaller than or equal to the wealth of the honest trader, and always strictly smaller than the wealth of the insider modeled
with the forward integral. However, the presence of the logarithmic utility changes this situation sharply. As we have shown herein, The Skorokhod insider attains the highest value in the case of constant parameters. Even if shorting is only forbidden for the Skorokhod insider, she still gets a higher value than the forward insider. In the case of time-dependent parameters, there is one particular case that can be solved and replicates the result of the Russo-Vallois forward integral, something without precedents in the case of risk-neutral traders. Moreover, for negative enough final values of the Brownian process, the ordinary trader can overcome both Skorokhod and forward integral insiders. A related feature, that the ordinary trader can overcome the insider one for certain paths in the case of time-dependent parameters, which could also be regarded as undesirable, was already studied in \cite{elizalde2022chances}, and identified as a consequence of the logarithmic utility. Now we have found that for certain driving Brownian paths, Skorokhod insiders cannot overcome ordinary traders; in particular, although the performance of Skorokhod insiders improves that of forward insiders under the logarithmic utility, it is unable to erase this feature.

Our results overall point to the fact that the interplay between stochasticity (through the introduction of a suitable stochastic integral) and nonlinearity
(through the introduction of a suitable utility function) still presents unexpected results within the realm of finance. A deeper understanding of the role
of Skorokhod integration in financial modeling could go through the computation of new explicit solutions to this type of stochastic differential equations,
something that has been quite elusive so far (in the present work, this translates into the necessity of assuming constant parameters and portfolios in order to fully approach the Skorokhod case); in fact, the solution to the Skorokhod differential equation present in Theorem \ref{euskorokhod} is new to the best of our knowledge. Also, the use of nonlinear utilities, which interacts well with classical stochastic calculus, yields new features that are not completely clear from a financial viewpoint when interrelated with anticipating calculus. Therefore, a possible future line of research is the development of a theory complementary to that of utilities and able to improve the models we have employed in this work. Our present results along with those in \cite{elizalde2022chances}
show that the uncritical use of these nonlinear utilities might shed consequences of difficult financial interpretation.

Another possible direction of future research is to establish the link between the present results and arbitrage opportunities, following for instance the line of \cite{corcuera2004}. And finally, it is important to remark that our results rely on mathematical assumptions on the portfolios that are of technical nature and due, as already mentioned, to the difficulty of dealing with Skorokhod calculus. Consequently, it would also be interesting to approach the problem we have herein analyzed with models more closely related to the financial practice. A possible way to achieve this is the one that departs from recent results on Skorokhod nonlinear stochastic differential equations \cite{leon2023stability}. Overall, it seems that extending the present research requires both modeling and mathematical developments that allow to account for those models.


\textbf{Data Availability Statement}\\
The information used to exemplify the performance of an insider using knowledge of the 2-Year U.S. Treasury Note Future was taken from https://www.cmegroup.com/markets/interest-rates/us-treasury/2-year-us-treasury-note.html.


\textbf{Acknowledgements}\\
This work has received funding from the European Union Horizon 2020 research and innovation programme under the Marie
Sk{\l}odowska-Curie grant agreement No. 777822.
It has been also partially supported by the Government of Spain (Ministerio de Ciencia e Innovaci\'on)
and the European Union through Projects PID2021-125871NB-I00, CPP2021-008644/
AEI/ 10.13039/ 501100011033/ Uni\'on Europea NextGenerationEU/ PRTR,
and TED2021-131844B-I00/
AEI/ 10.13039/ 501100011033/ Uni\'on Europea NextGenerationEU/ PRTR.


\bibliographystyle{acm}
\bibliography{APO_bib}

\end{document}